\documentclass[a4paper,12pt]{amsart}

\usepackage{lmodern}

\usepackage[T1]{fontenc}        % Codage des fontes
\usepackage[latin1]{inputenc}   % Codage du fichier
\usepackage[english]{babel}
\usepackage{amssymb,amsfonts,amsthm,amsmath,latexsym}
\usepackage{enumerate, caption, array, esint}
\usepackage{hyperref}

\usepackage[margin=1.1in]{geometry}

\theoremstyle{plain}
\newtheorem{theoreme}{Theorem}%[section]
\newtheorem{coro}{Corollary}%[section]
\newtheorem{lemme}{Lemma}%[section]
\newtheorem{prop}{Proposition}%[section]
\newtheorem{thmext}{Theorem}%[section]

\theoremstyle{definition}

\theoremstyle{remark}
\newtheorem*{remarque}{Remark}

\newcommand\numberthis{\stepcounter{equation}\tag{\theequation}}

\renewcommand{\setminus}{\smallsetminus}
\newcommand{\ssum}[1]{\sum_{\substack{#1}}}

\newcommand{\e}{{\rm e}}
\newcommand{\dd}{{\rm d}}

\newcommand{\hT}{T}

\newcommand{\cC}{{\mathcal C}}

\newcommand{\cD}{{\mathcal D}}

\newcommand{\cV}{{\mathcal V}}

\newcommand{\cK}{{\mathcal K}}
\newcommand{\cH}{{\mathcal H}}

\newcommand{\cR}{{\mathcal R}}
\newcommand{\cU}{{\mathcal U}}

\newcommand{\cu}{{\varrho}}

\newcommand{\ee}{{\varepsilon}}

\newcommand{\tphi}{\tilde{\phi}}
\newcommand{\tsigma}{\tilde{\sigma}}

\newcommand{\bx}{{a}}
\newcommand{\by}{{b}}
\newcommand{\bz}{{z}}
\newcommand{\sfa}{{ a}}
\newcommand{\sfb}{{ b}}

\newcommand{\bfD}{{\mathbf D}}
\newcommand{\bfN}{{\mathbf N}}
\newcommand{\bfZ}{{\mathbf Z}}
\newcommand{\bfQ}{{\mathbf Q}}
\newcommand{\bfR}{{\mathbf R}}
\newcommand{\bfC}{{\mathbf C}}
\newcommand{\bfUn}{{\mathbf 1}}

\newcommand{\frakH}{{\mathfrak H}}
\newcommand{\hess}{{\rm Hess}}
\newcommand{\Prob}{{\rm Prob}}

\newcommand{\CpR}{{\bfC\smallsetminus\bfR_-}}

\newcommand{\fr}{{\mathfrak r}}
\newcommand{\fs}{{\mathfrak s}}

\renewcommand{\tilde}{\widetilde}

\newcommand{\tg}{{\tilde g}}

\newcommand{\card}{{\rm card\ }}

\newcommand{\ubar}{{\bar u}}
\newcommand{\floor}[1]{{\left\lfloor {#1} \right\rfloor}}

\renewcommand\Re{\operatorname{\mathfrak{Re}}}
\renewcommand\Im{\operatorname{\mathfrak{Im}}}

\DeclareMathOperator{\argmin}{argmin}

\numberwithin{equation}{section}

\title{On the average distribution of divisors of friable numbers}
\author{Sary Drappeau}
\date{\today}

\address{I2M, Aix-Marseille Université
\\ 163 avenue de Luminy
\\ 13009 Marseille, France}
\email{sary-aurelien.drappeau@univ-amu.fr}

\begin{document}

\maketitle

\begin{abstract}
A number is said to be~$y$-friable if it has no prime factor greater than~$y$. In this paper, we prove a central limit theorem on average for the distribution of divisors of~$y$-friable numbers less than $x$, for all~$(x,y)$ satisfying $2\leq y\leq \e^{(\log x)/(\log\log x)^{1+\ee}}$. This was previously known under the additional constraint~$y\geq \e^{(\log\log x)^{5/3+\ee}}$, by work of Basquin. Our argument relies on the two-variable saddle-point method.
\end{abstract}

%\tableofcontents

\section{Introduction}

An integer~$n\geq 1$ is said to be~$y$-friable, or~$y$-smooth, if its greatest prime factor~$P(n)$ is less than or equal to~$y$, with the convention~$P(1)=1$. We denote
\[ S(x, y) := \big\{n\leq x :\quad P(n)\leq y\big\}, \]
\[ \Psi(x, y) := \card S(x, y) .\]
Friable integers are a recurrent object in analytic number theory: we refer the reader to the surveys~\cite{survey-ht, survey-granville, survey-moree} for an overview of recent results and applications. An important aspect of results about friable numbers is their uniformity with respect to~$y$. The difficulty in this context is the fact that~$y$-friable numbers tend to rarefy very rapidly -- much more so than what would be expected from sieve heuristics, for instance. In this respect, analytic methods have proven to be very effective. The object of this paper is to study, using these analytic methods, the distribution of divisors of friable numbers on average.

For any~$n\geq 1$, define~$D_n$ to be the random variable taking the value~$\log d$ where~$d$ is chosen among the~$\tau(n)$ divisors of~$n$ with uniform probability. It was shown by Tenenbaum~\cite{Tene-Div2} that~$D_n/\log n$ does not converge in law on any sequence of integers~$n$ of positive upper density in~$\bfN$. However it can be expected that the discrepancies arising from the erratic behaviour of the multiplicative structure are smoothed out upon averaging over~$n$. When one averages over all the integers, this question was settled by Deshouillers, Dress and Tenenbaum~\cite{DDT} who established
\begin{equation}
\frac1x\sum_{n\leq x}\Prob(D_n\leq t \log n) = \frac2\pi\arcsin\sqrt{t} + O\Big(\frac1{\sqrt{\log x}}\Big), \qquad (t\in [0, 1]). \label{estim-DDT}
\end{equation}
The error term here is optimal if one seeks full uniformity with respect to~$\gamma$. See also~\cite{BM-DDT} for a generalization (where one changes the probability measure one puts on the divisors).

Expectedly, the analog problem for friable numbers has a different structure. Choosing a divisor of~$n$ at random is equivalent to choosing an integer~$k\in\{0, \ldots, \nu\}$ uniformly at random for every factor~$p^\nu\|n$ appearing in the decomposition of~$n$. We may thus write
\begin{equation}
D_n = \sum_{p^\nu\| n} D_{p^\nu}\label{eq:D-xi}
\end{equation}
where the variables~$D_{p^\nu}$ on the right-hand side are independent. As is known from work of Alladi~\cite{Alladi} and Hildebrand~\cite{Hilde-omega}, the number~$\omega(n)$ of prime factors of~$n\in S(x, y)$ typically tends to grow with~$n$, in such a way that we may expect the sum~\eqref{eq:D-xi} to satisfy the central limit theorem. We are therefore led to the prediction that
\begin{equation}
\Prob\big(D_n\geq \tfrac12\log n + v \cu_n\big) \approx \Phi(v) := \int_v^{\infty} \frac{\e^{-z^2/2}\dd z}{\sqrt{2\pi}}\label{eq:Dn-TCL}
\end{equation}
for all fixed~$v\in\bfR$ and almost all integers~$n\in S(x, y)$, where~$\cu_n$ denotes the standard deviation of~$D_n$, given by
\begin{equation}
\cu_n^2 = \sum_{p^\nu \| n} \frac{\nu(\nu+2)}{12}(\log p)^2.\label{eq:def-V}
\end{equation}
La Bretèche and Tenenbaum~\cite[Corollaire~2.2]{BT2002} consider the case of the primorial number~$N_1(y) :=\prod_{p\leq y}p$ (which is the largest square-free $y$-friable number). They obtain
\begin{equation}
\Prob\big(D_{N_1}\geq \tfrac12\log N_1 + v\cu_{N_1(y)} \big) = \Phi(v)\Big\{1 + O\Big(\frac{1+v^4}{y/\log y}\Big)\Big\}\label{estim-BT-N1}
\end{equation}
for~$0\leq v\ll(y/\log y)^{1/4}$. Note that~$\cu_{N_1(y)}^2 \sim (y\log y)/4$. We emphasize that there is no average over the integers under study. Another related example considered recently by Tenenbaum~\cite[Corollaire~1.4]{Tene-Ultra}, is the case of~$N_2(y) := \prod_{p\leq y}p^{\floor{(\log y)/\log p}}$. There, Tenenbaum obtains an analogous result to~\eqref{estim-BT-N1}.

Such a law obviously does not hold for all~$y$-friable numbers, as illustrated by the example of~$N_3(y) = 2^{\floor{y/\log 2}}$ (which is roughly of the same size as~$N_1$ and~$N_2$, but for which~$D_{N_3}$ converges to the uniform law). It is therefore natural to ask what the output is, if we on average over friable numbers. One option would be to study the average
$$ \frac1{\Psi(x, y)}\sum_{n\in S(x, y)} \Prob(D_n\geq \tfrac12\log n + v\cu_n). $$
However, a more interesting variant is deduced from observing that an additive function of~$n$ naturally appears in the formula~\eqref{eq:def-V}. A fundamental result in probabilistic number theory, the Tur\'{a}n-Kubilius inequality, developped in the context of friable numbers by La Bretèche and Tenenbaum~\cite{BT-TK}, ensures the existence of a quantity~$\cu(x, y)$ \emph{independent of~$n$} such that
\begin{equation}
\cu_n \sim \cu(x, y)\label{eq:cun-cu}
\end{equation}
for a relative proportion~$1 + o(1)$ of integers~$n\in S(x, y)$, when~$y\to\infty$ and~$y=x^{o(1)}$. The exact definition of~$\cu(x, y)$ involves the \emph{saddle-point}~$\alpha(x, y)$, defined as the only positive solution to the equation
$$ \sum_{p\leq y} \frac{\log p}{p^\alpha-1} = \log x. $$
Then the approximation~\eqref{eq:cun-cu} holds with
$$ \cu = \cu(x, y) := \Big(\frac14\sum_{p\leq y}\frac{p^\alpha-\frac13}{(p^\alpha-1)^2}(\log p)^2\Big)^{1/2}. $$
We will prove below that
$$ \cu(x, y)^2 \sim (\log x)(\log y)\big(\frac14 + \frac{\log x}{6y}\big) \qquad\qquad (y \to \infty,\ y=x^{o(1)}). $$
In view of the above, we consider for~$v\in\bfR$ the quantity
\begin{align*}
D(x, y ; v) :=\ & \frac1{\Psi(x, y)}\sum_{n\in S(x, y)}\Prob(D_n\geq \tfrac12\log n + v\cu ) \numberthis\label{def-D} \\
=\ & \frac1{\Psi(x, y)}\sum_{n\in S(x, y)}\frac1{\tau(n)}\ssum{d|n\\d\geq n^{1/2}\e^{v\cu}}1 \qquad\qquad (2\leq y\leq x,\ v\in\bfR).
\end{align*}
The asymptotic behaviour of~$D(x, y ; v)$ was studied previously by Basquin~\cite{Basquin}\footnote{To be precise, in~\cite{Basquin}, Basquin studies the slight variant where~$\e^{v\cu}$ is replaced by~$n^{v\cu/\log x}$. This change does not affect the estimate of Theorem~\ref{thm-basquin}.} for relatively large values of~$y$. There, Basquin quantifies the shift from the arcsine law~\eqref{estim-DDT} to a contracted normal law similar to~\eqref{estim-BT-N1}: we refer the reader to~\cite[Théorème~1.1]{Basquin} for more details about this transition. We shall focus on the gaussian behaviour for small values of~$y$: let
\[ u := (\log x)/\log y, \qquad \ubar := \min\{u, \pi(y)\}, \]
where~$\pi(y)$ denotes the counting function of primes. Then Theorem~1.1 and Corollary~1.3 of~\cite{Basquin} (along with~\cite[equation~(7.19)]{HT86} to relate~$\cu$ with the quantity~$\xi'(u)$ involved there) imply the following.
\begin{thmext}\label{thm-basquin}
Then for all~$\ee>0$ and all~$x$ and~$y$ satisfying
\begin{equation}\label{domaine-He}\tag{$H_\ee$} \exp\{(\log \log x)^{5/3+\ee}\} \leq y \leq x ,
\end{equation}
we have
\begin{equation}\label{estim-basquin} D(x, y ; v) = \Phi(v) + O_\ee\Big(\frac1u + \frac1{\sqrt{\log y}} + \frac{\log(u+1)}{\log y}\Big) \qquad (v\in\bfR). \end{equation}  
\end{thmext}
The range of validity in~$x$ and~$y$ here is inherent to the method used, which is based on the ``indirect'' saddle-point method (see also~\cite{Saias}). The purpose of the present work is to introduce a variant of the two-variable (direct) saddle-point method which allows us to obtain a significant improvement of the range of validity and of the error term in Theorem~\ref{thm-basquin}.
\begin{theoreme}\label{thm-pcp}
% Let~$\cu = \cu(x, y)$ be defined for~$2\leq y\leq x$ by~\eqref{def-cu} below. In particular, we have~$\cu \asymp (\log x)^2/\ubar$, and as~$\ubar\to\infty$,
% \[ \begin{cases}\cu\sim \tfrac14(\log x)\log y &\text{if }y/(\log x) \to\infty, \\ \cu\sim \tfrac16(\log x)^2/\pi(y) &\text{if }y/(\log x)\to 0. \end{cases} \]
Let~$\ee>0$. Whenever~
\begin{equation}\label{eq:domaine-th1}\tag{$G_\ee$}
x\geq 3, \qquad 2\leq y \leq \e^{(\log x)/(\log\log x)^{1+\ee}},
\end{equation}
and~$0\leq v \ll (\ubar)^{1/4}$, we have
\begin{equation}
\label{pcp-1}
D(x, y ; v) = \Phi(v)\Big\{1 + O_\ee\Big(\frac{1+v^4}{\ubar}\Big)\Big\} .
\end{equation}
\end{theoreme}
The condition~$y \leq \e^{(\log x)/(\log\log x)^{1+\ee}}$ is purely technical. For~$x$ and~$y$ in the complementary range~$u\leq (\log\log y)^{1+\ee}$, the Gaussian approximation is less relevant and the methods of~\cite{Basquin} are better suited.

The range~$v\ll (\ubar)^{1/4}$ is the natural range of validity of the Gaussian approximation. As is typically the case in large deviation theory, one could expect an asymptotic formula to hold in the range~$v\ll (\ubar)^{1/2-\ee}$ by adding correction terms to the exponent~$z^2/2$ in the definition~\eqref{eq:Dn-TCL} of~$\Phi(v)$. We prove that such is indeed the case.
\begin{theoreme}\label{thm:pcp-2}
Let~$(x, y)\in (G_\ee)$. There exists a sequence of numbers~$(b_j(x, y))_{j\geq 0}$ satisfying
$$ b_0(x, y) = -1/2, \qquad b_j(x, y) \ll_j (\ubar)^{-j}, $$
such that the following holds. Let~$k\geq 1$ and
$$ v_{\text{max}} \asymp (\ubar)^{k/(2k+2)} $$
be given, and assume that~$0\leq v\leq v_{\text{max}}$. Letting
\begin{equation}
\cR_k(z) = \cR_k(x, y ; z) := \sum_{j=0}^{k-1} b_j(x, y) z^{2(j+1)} \qquad (z\geq 0),\label{eq:def-Rk}
\end{equation}
we have
\begin{equation}
\label{pcp-2}
\begin{aligned}
D\big(x, y ; v\big) = \Big\{1 + O_k\Big(\frac{1+v^2}{\ubar}+\frac{v^{2(k+1)}}{(\ubar)^k}\Big)\Big\}
\int_v^{2v_{\text{max}}} \e^{\cR_k(z)}\frac{\dd z}{\sqrt{2\pi}} .
\end{aligned}
\end{equation}
\end{theoreme}
\begin{remarque}
Note that~$\cR_k(z) = -z^2/2 + O(z^4/\ubar)$, which explains the shape of the error term in~\eqref{pcp-1}.

The coefficients~$b_j(x, y)$ for~$j\geq 1$ could be expressed, if one wished, as an explicit but complicated expression involving sums over primes less than~$y$ and the saddle-point~$\alpha(x, y)$ defined below. As~$\ubar\to\infty$, they can be approximated by elementary expressions involving~$x$ and~$y$, in the same shape as formula~\eqref{estim-phi2} below. We refrain to do so here.
\end{remarque}

\subsection{The saddle-point method}

We now recall the explanation for the limitation on~$y$ in the estimate of Basquin~\cite{Basquin}. The range~\eqref{domaine-He} is classical in the study of friable numbers: it is typically linked to the approximation of~$\Psi(x, y)$ by Dickman's function\footnote{Dickman's function~$\rho$ is the unique continuous function on~$\bfR_+$ which is differentiable on~$(1, \infty)$, satisfies~$\rho(u)=1$ for~$u\in[0, 1]$, and~$u\rho'(u)+\rho(u-1)=0$ for~$u>1$. We have~$\rho(u)=u^{-u+o(u)}$ as~$u\to\infty$.}~$\rho$:
\begin{equation}\label{estim-hildebrand} \Psi(x, y) = x\rho(u)\Big\{1+O_\ee\Big(\frac{\log(u+1)}{\log y}\Big)\Big\} \qquad ((x, y)\in (H_\ee)). \end{equation}
This estimate is a theorem of Hildebrand~\cite{Hild86}, improving in particular De Bruijn's work~\cite{dB1951}. The range~$(H_\ee)$ is tighly linked to the best known error term in the prime number theorem: it was shown by Hildebrand~\cite{Hild84} that if one could prove the weaker estimate~$\Psi(x, y) = x\rho(u)\exp\{O(y^\ee)\}$ for~$y\geq(\log x)^{2+\ee}$, for all fixed~$\ee>0$, then the Riemann hypothesis would follow.

In many applications however, including that of interest here, one seeks a control on the local variations of~$\Psi(x, y)$ with respect to~$x$, rather than a control of~$\Psi(x, y)$ itself. By ``local variations'' we mean, for instance, quantities of the shape~$\Psi(x/d, y)/\Psi(x, y)$ for relatively small~$d\geq 1$. The saddle-point approach to estimating~$\Psi(x, y)$, developped by Hildebrand and Tenenbaum~\cite{HT86}, is very suitable for such applications: it enabled very substantial progress to be made in the last decades regarding the uniformity with respect to~$y$, for example in friable analogs of the Tur\'{a}n--Kubilius inequality~\cite{BT-TK} or distribution of friable numbers in arithmetic progressions~\cite{Sound2008}.

We now recall Hildebrand and Tenenbaum's result. When~$2\leq y\leq x$, the saddle-point~$\alpha(x, y)$ is defined as the positive real number satisfying
\begin{equation}\label{def-alpha} \sum_{p\leq y}\frac{\log p}{p^\alpha-1} = \log x .\end{equation}
It is therefore the positive number optimizing Rankin's simple but remarkably efficient upper bound
\begin{equation}\label{majo-rankin}
\Psi(x, y) \leq \min_{\sigma>0} \zeta(\sigma, y)x^\sigma, 
\end{equation}
where
\[ \zeta(s, y) := \prod_{p\leq y}(1-p^{-s})^{-1} = \sum_{P(n)\leq y}n^{-s}\quad (\Re(s)>0). \]
Here and in what follows, the letter~$p$ always denotes a prime number. As was pointed out in~\cite{HT86}, the point~$s=\alpha(x, y)$ is a saddle point for the Mellin transform~$x^s\zeta(s, y)$ relevant to~$\Psi(x, y)$:
\[ \Psi(x, y) = \frac1{2\pi i}\int_{\sigma-i\infty}^{\sigma+i\infty}x^s\zeta(s, y)\frac{\dd s}s, \qquad (x\not\in\bfN, \sigma>0). \]
Letting
\[ \phi_2(s, y) = \sum_{p\leq y}\frac{(\log p)^2p^s}{(p^s-1)^2} \qquad (\Re(s)>0), \]
they obtain for~$2\leq y\leq x$ the following estimate~\cite[Theorem~1]{HT86}~:
\begin{equation}\label{estim-psi-col}
\Psi(x, y) = \frac{\zeta(\alpha, y)x^\alpha}{\alpha\sqrt{2\pi \phi_2(\alpha, y)}}\Big\{1 + O\Big(\frac1\ubar\Big)\Big\}.
\end{equation}
The denominator~$\alpha\sqrt{2\pi\phi_2(\alpha, y)}$ in~\eqref{estim-psi-col} may be estimated using~\cite[Theorem~2.(ii)]{HT86}. We have
\begin{equation}\label{estim-alpha}
\alpha(x, y) = \frac{\log(1+y/(\log x))}{\log y}\Big\{1 + O\Big(\frac{\log\log(1+y)}{\log y}\Big)\Big\} ,\end{equation}
\begin{equation}\label{estim-phi2}
\phi_2(x, y) = \Big(1+\frac{\log x}y\Big)(\log x) \log y \Big\{1 + O\Big(\frac1{\log(u+1)}+\frac1{\log y}\Big)\Big\}
.\end{equation}
However, the question of approximating~$\zeta(\alpha, y)x^\alpha$ up to an factor~$(1+o(1))$ by a smooth and explicit function of~$x$ and~$y$ -- for instance, in terms of the Dickman function~$\rho$, is tightly related to the error term in the prime number theorem. In a way, $\alpha$ encodes the irregularities in the distribution of prime numbers that prevent us from having a smooth, explicit estimate for~$\Psi(x, y)$ when~$(x, y)\not\in (H_\ee)$ for all~$\ee>0$.

On the other hand, the local variations of~$\alpha(x, y)$ with respect to~$x$ are relatively well controlled~: such local estimates were obtained by La Bretèche--Tenenbaum~\cite{BT2005}. We note however that at the current state of knowledge, when~$(x, y)\not\in(H_\ee)$, we are not able to deduce from them an equivalent \textit{e.g.}\ of the quantity~$\sqrt{\Psi(x^2, y)}/\Psi(x, y)$, or the quantity
\begin{equation}
  \label{sum-1surtau}
  \frac1{\Psi(x, y)}\sum_{n\in S(x, y)}\frac1{\tau(n)}.
\end{equation}
This is hinted, for instance, by the fact that the error terms of~\cite[Théorème~2.4]{BT2005}, which result from the estimation of~$\Psi(x/d, y)/\Psi(x, y)$, are of the same size as the main term if~$d=\sqrt{x}$. Note that if~$y\geq (\log x)^3$, say, the saddle-point relevant to the sum in~\eqref{sum-1surtau} is roughly of the same size as~$\alpha(x^2, y)$ (because~$1/\tau(p)=1/2$ for prime~$p$). The issue at hand when studying~$D(x, y ; \gamma)$ is precisely the estimation of such sums as the one in~\eqref{sum-1surtau}; in our case however, as will be apparent, the upper bound on~$n$ will be roughly of size~$x^{1/2+o(1)}$, and the relevant saddle-point will indeed be well-approximated, to some extent, by~$\alpha(x^{1+o(1)}, y)$.

\subsection{A truncated convolution and the two-variable saddle-point method}

We now sketch our proof of Theorem~\ref{thm-pcp}. Inverting summations yields
\[ D(x, y ; \gamma) = \frac{S}{\Psi(x, y)}, \]
where
\[ S := \underset{\substack{P(nd)\leq y \\ nd\leq x, \\d^{1/2}\geq n^{1/2}\e^\gamma}}{\sum\sum}\frac1{\tau(nd)}. \]
The obvious approach here consists in first approximating the sum over~$n$ by a ``nice'' function of~$d$, and then estimating the remaining sum over~$d$. This is the method followed~\textit{e.g.}~in~\cite{Basquin}. There, one relies on estimates for friable sums of multiplicative functions from~\cite{Smida}, which are a generalization of~\eqref{estim-hildebrand}. These however are still subject to the limitation~$(x, y)\in(H_\ee)$.

One could presumably follow the same strategy by using the estimate~\eqref{estim-psi-col} along with local estimates for the saddle-point. The need for uniformity in~$d$ for the estimation of the inner sum, however, is likely to produce significant technical complications due to the dependence of the summand on the multiplicative structure of~$d$. Here instead we study the double sum as a whole by applying the Perron formula twice, which yields
\begin{equation}\label{perron2} S = \frac1{(2\pi i)^2}\int_{\sigma-i\infty}^{\sigma+i\infty}\int_{\kappa-i\infty}^{\kappa+i\infty} x^{s} \e^{-\gamma w} F_y(s+w/2, s-w/2)\frac{\dd w}w\frac{\dd s}s, \qquad (2\sigma > \kappa>0), \end{equation}
provided~$x\not\in\bfN$ and~$\e^{2\gamma}\not\in\bfQ$. Here~$F_y(s, w)$ is the Dirichlet series relevant to our problem
\[ F_y(s, w) := \sum_{P(nd)\leq y}\frac1{\tau(nd)n^sd^w}, \qquad (\Re(s), \Re(w)>0), \]
and~$\gamma = v \cu$. One wishes to apply the saddle-point method for the double-integral in~\eqref{perron2}. A linear change of variables yields
$$ S = \frac2{(2\pi i)^2}\int_{\sigma-i\infty}^{\sigma+i\infty}\int_{\kappa+\sigma-i\infty}^{\kappa+\sigma+i\infty} x^{(s+w)/2} \e^{\gamma(w-s)}F_y(s, w) \frac{\dd w\dd s}{(s-w)(s+w)}. $$
The effect of the factor~$1/(s-w)$ cannot be fully neglected; although a direct analysis would likely be possible (as in~\cite[Corollary 2.2]{BT2002}), we circumvent this issue by truncating off values of~$s$ and~$w$ with large imaginary parts, and differentiating with respect to~$v$. Therefore, for some~$T>0$ of a suitable size and for some optimal choice of~$(\sigma, \kappa)$ (depending on~$x$, $y$ and~$\gamma$), one wishes to estimate
\[ \frac{2 \cu}{(2\pi i)^2}\int_{\sigma-iT}^{\sigma+iT}\int_{\kappa+\sigma-iT}^{\kappa+\sigma+iT} x^{(s+w)/2}\e^{\gamma(w-s)}F_y(s, w) \frac{\dd w\dd s}{s + w}. \]
The integrals there can be analyzed by the saddle-point method, which eventually yields the expected approximation~$\Psi(x, y)\e^{-v^2/2}/\sqrt{2\pi}$.

Finally, we note that very recently Robert and Tenenbaum~\cite{RT} used a variant of the two-variable saddle-point method to study the distribution of integers with small square-free kernel. Compared with theirs, our setting is simplified by the fact that the series~$F_y(s, w)$ is symmetric and to some extent comparable to~$\zeta(s, y)^{1/2}\zeta(w, y)^{1/2}$ (for the study of which we can use the work of Hildebrand and Tenenbaum~\cite{HT86}).

\subsection{Acknowledgments}

The author was supported by a CRM-ISM Postdoctoral fellowship. The author is grateful to Régis de la Bretèche, Andrew Granville and Gérald Tenenbaum for helpful comments on an earlier version of this manuscript.

\section{Preliminary remarks and notations}

We will keep throughout the notation
\[ s = \sigma+i\tau, \quad w=\kappa+it, \quad ((\sigma,\tau,\kappa,t)\in\bfR^4). \]
We write~$A\ll B$ or~$A=O(B)$ whenever~$A$ and~$B$ are expressions where~$B$ assumes non-negative values, and there exists a positive constant~$C$ such that~$|A|\leq CB$ uniformly. The constant~$C$ may depend on various parameters, which are then displayed in subscript (\textit{e.g.}~$A\ll_\ee B$ if the constant depends on~$\ee$). Moreover, the letters~$c_1, c_2, \ldots$ designate positive constants, which are tacitly assumed to be absolute, unless otherwise specified.

At various places in our arguments, functions such as~$z\mapsto 1/(\log z) - 1/(z-1)$ are involved, which are regular at some particular point of their domain of definition, where the explicit expression diverges (here~$z=1$). It will be implicit that one should consider the holomorphic extension at said point.

Finally, every instance of the complex logarithm function we consider is, unless otherwise specified, the principal determination defined on~$\CpR$. For all~$r>0$ and any function~$f$ defined on~$\CpR$, we denote~$f(-r+0i):=\lim_{\ee\to0+}f(-r+i\ee)$ and similarly~$f(-r-0i):=\lim_{\ee\to0+}f(-r-i\ee)$, whenever those limits exist.

\section{Saddle-point estimates for~$\zeta(s, y)$}\label{section-sigma}

For all~$k\in\bfN$,~$s\in\bfC$ with~$\Re(s)>0$ and~$y\geq 2$, we define
\[ \phi_0(s, y) := \log \zeta(s, y) = -\sum_{p\leq y}\log(1-p^{-s}), \quad \phi_k(s, y) := \frac{\partial^k\phi_0}{\partial s^k}(s, y), \quad \tphi_k(s, y) := \sum_{p\leq y}\frac{(\log p)^k}{(p^s-1)^k}, \]
\[ \sigma_k := \phi_k(\alpha, y), \qquad \tsigma_k := \tphi_k(\alpha, y). \]
Bear in mind that~$\sigma_k$ and~$\tsigma_k$ depend on~$x$ and~$y$, the values of which will be clear from the context. In particular, by the definition of~$\alpha$,
\begin{equation}\label{def-sigmas} \sigma_1 =-\log x, \qquad \sigma_2 = \sum_{p\leq y}\frac{(\log p)^2p^\alpha}{(p^\alpha-1)^2}, \qquad \tsigma_2 = \sum_{p\leq y}\frac{(\log p)^2}{(p^\alpha-1)^2}.\end{equation}

We quote the following useful estimates on~$\alpha(x, y)$ and~$\phi_k(\alpha, y)$ from Theorem~2 and Lemmas~2, 3 and~4 of~\cite{HT86}. They will be implicitly used thoughout our argument. Uniformly for~$2\leq y\leq x$, we have
\[ \sigma_k \asymp (u\log y)^k(\ubar)^{1-k}, \quad \alpha \asymp \frac{\ubar}{u\log y} \quad (y\ll \log x), \qquad \alpha \gg \frac1{\log y} \quad (y\gg\log x), \]
\[ (1-\alpha)\log y \ll \log\ubar, \qquad \sqrt{\ubar}\ll \alpha\sqrt{\sigma_2}\ll \min\{\sqrt{\ubar}\log y, \sqrt{y/\log y}\}. \]
We will also require the following two bounds, which are corollaries of the calculations of~\cite[page~281]{HT86}. We have
\begin{equation}
\label{HT-p281}
\begin{aligned}
\int_{(\ubar)^{2/3}/(\log x)}^\infty\Big(1+\frac{\tau^2\sigma_2}{y/(\log y)}\Big)^{-cy/(\log y)}\dd\tau&\ \ll \frac1{\ubar\sqrt{\sigma_2}}, \\ \int_{0}^\infty\Big(1+\frac{t^2\sigma_2}{y/(\log y)}\Big)^{-cy/(\log y)}\dd t&\ \ll \frac1{\sqrt{\sigma_2}}.
\end{aligned}
\end{equation}
Regarding~$\tphi_k(x, y)$, using prime number sums in the same way as~\cite[Lemma~4 and~13]{HT86}, we deduce that
\[ \tphi_k(\sigma, y) \ll_k |\phi_k(\sigma, y)| \qquad (k\geq 2,\ \sigma>0,\ y\geq 2). \]
Note that we trivially have~$\tsigma_2\leq \sigma_2$. The next lemma relates more precisely the two quantities.
\begin{lemme}
As~$y, u\to\infty$,
\begin{equation}\label{estim-tphi2}
\frac{\tsigma_2}{\sigma_2} = \frac1{1+(y/\log x)} + o(1).
\end{equation}
\end{lemme}
\begin{proof}
When~$\alpha\geq 0.6$, we certainly have~$y/(\log x)\to\infty$ as well as~$\tsigma_2 = O(1)$ and~$\sigma_2\to\infty$, so that the desired estimate holds. We may thus assume that~$\alpha<0.6$.

Let~$\ee\in(0, 1/10]$. By~\cite[Lemma~3]{HT86}, we have~$\phi_2(\alpha, y) \asymp y^{1-\alpha}\log y$ whenever~$1/(\log y) \ll \alpha \leq 0.6$. The same conditions are satisfied when one replaces~$y$ by~$y^{1/2}$; we deduce
\begin{equation}\label{majo-phi2-tail} \phi_2(\alpha, y^{1/2}) \asymp y^{-(1-\alpha)/2} \phi_2(\alpha, y) \leq y^{-\alpha/3}\phi_2(\alpha, y). \end{equation}
Suppose first that~$y\geq (1/\ee)\log x$. Then~$\log(1/\ee)/\log y \ll \alpha < 0.6$, and we have
\[ \tphi_2(\alpha, y) = \sum_{p\leq y^{1/2}}\frac{(\log p)^2}{(p^\alpha-1)^2} + \sum_{y^{1/2}<p\leq y}\frac{(\log p)^2}{(p^\alpha-1)^2} \leq \phi_2(\alpha, y^{1/2}) + y^{-\alpha/2}\phi_2(\alpha, y) \ll \ee^c\phi_2(\alpha, y) \]
for some absolute constant~$c>0$, because of our assumption on~$\alpha$.

Assume next that~$y\leq \ee\log x$. Then~$\alpha\ll \ee/\log y$ and~$p^\alpha = 1+O(\ee)$ uniformly for~$p\leq y$, so that
\[ \tphi_2(\alpha, y) = \{1+O(\ee)\}\phi_2(\alpha, y). \]

Finally assume that~$y=t\log x$ where~$t$ varies inside~$(\ee, 1/\ee)$ and let~$\ubar\to\infty$. Then we have~$\alpha\sim\log(1+t)/\log y$, so that~$y^\alpha\sim_\ee(1+t)$ (the decay of the implied~$o(1)$ there may depend on~$\ee$). Evaluating the sum over primes defining~$\phi_2(\alpha, y)$ using~\cite[Lemma~13]{HT86}, we have
\[ \phi_2(\alpha, y) = \frac{1+o(1)}{(1-y^{-\alpha})^2}\int_2^y\frac{(\log z)\dd z}{z^\alpha} +O(1)\sim_\ee \frac{y^{1-\alpha}\log y}{(1-y^{-\alpha})^2} \sim_\ee (t^{-1}+t^{-2})y\log y. \]
The same set of calculations show that, on the other hand,
\[ \tphi_2(\alpha, y) = \frac{1+o(1)}{(1-y^{-\alpha})^2}\int_2^y\frac{(\log z)\dd z}{z^{2\alpha}} + O(1) \sim_\ee \frac{y^{1-2\alpha}\log y}{(1-y^{-\alpha})^2} \sim_\ee t^{-2}y\log y. \]
We deduce~$\tphi_2(\alpha, y) \sim_\ee (1+t)^{-1}\phi_2(\alpha, y)$.

Grouping our estimates, we have in any case
\[ \limsup_{\ubar\to\infty}\Big(\frac{\tsigma_2}{\sigma_2}-\frac1{1+(y/\log x)}\Big) \ll \ee^c \]
for some absolute~$c>0$ and all~$\ee>0$, and we conclude by letting~$\ee\to0$.
\end{proof}

Having the above facts at hand, we let~$\cu = \cu(x, y)$ be defined for~$2\leq y\leq x$ by
\begin{equation}
  \label{def-cu}
  \cu := \tfrac12(\sigma_2-\tsigma_2/3)^{1/2} \asymp (\log x)/\sqrt{\ubar}.
\end{equation}
As~$\ubar\to\infty$, we therefore have
\[ \cu^2 \sim (\log x) \log y \big(\frac14 + \frac{\log x}{6y}\big). \]

\section{Lemmas}

The following lemma is a truncated Perron formula suited for sparse sequences, \textit{cf.}~\cite[Exercices II.2.2 and II.2.3]{Tene2007}. Let
\[ K(\tau) := \max\{0, 1-|\tau|\} \qquad (\tau\in\bfR). \]
\begin{lemme}\label{lemme-perron}
Let~$(a_n)$ be any sequence of complex numbers, and assume that the series
\[ F(s) := \sum_{n\geq 1}\frac{a_n}{n^s} \]
is absolutely convergent on the half-plane~$\Re(s)>\sigma_0$ for some~$\sigma_0>0$. For all such~$s$, let~$F_0(s) := \sum_{n\geq 1}|a_n|n^{-s}$.
Then for all~$x\geq 2$, $\sigma>\sigma_0$ and~$T\geq 2$, we have
\[ \sum_{n\leq x}a_n = \frac1{2\pi i}\int_{\sigma-iT}^{\sigma+iT}F(s)\frac{x^s\dd s}s + O\Big(\frac{x^\sigma}{\sqrt{T}}\Big\{F_0(\sigma) + \int_{-\sqrt{T}}^{\sqrt{T}}x^{i\tau}F_0(\sigma+i\tau)K(\tau/\sqrt{T})\dd\tau\Big\}\Big) .\]
\end{lemme}
\begin{remarque} The integral is the error term is a non-negative real number, as is apparent from the proof. \end{remarque}
\begin{proof}
The estimate follows classically from the formula, valid for all~$z>0$,
\begin{equation}\label{perron-z}
\frac1{2\pi i}\int_{\sigma-iT}^{\sigma+iT}\frac{z^s\dd s}s = \bfUn_{z\geq1} + O\big(z^\sigma\min\{1, (T|\log z|)^{-1}\}\big) \ll z^\sigma.
\end{equation}
Indeed the error term is~$O\big(z^\sigma\{T^{-1/2}+\bfUn_{|\log z|\leq T^{-1/2}}\}\big)$, and we have
\begin{equation}
\bfUn_{|\log z|\leq T^{-1/2}} \ll \Big(\frac{\sin(\sqrt{T}(\log z)/2)}{\sqrt{T}(\log z)/2}\Big)^2 = \int_{-1}^1 z^{i\tau\sqrt{T}}K(\tau)\dd \tau.\label{majo-perron-K}
\end{equation}
We then specialize at~$z=x/n$ and sum over~$n$ against the coefficients~$a_n$.
\end{proof}

\subsection{Basic properties of~$F_y(s, w)$}

Let
\[ \cH := \{s \in \bfC :\ \Re(s)>0\}, \qquad \cU := \{\bz\in\bfC : |\bz|<1\}. \]
For all~$(s, w)\in\cH^2$, we write
\[ F_y(s, w) := \sum_{P(nd)\leq y}\frac1{\tau(nd)n^sd^w}. \]
Note that we have the Euler product expansion
\[ F_y(s, w) = \prod_{p\leq y}\Big(\sum_{k,\ell \geq 0}\frac{p^{-ks - \ell w}}{k+\ell+1}\Big) = \prod_{p\leq y}\Big(\frac{\log(1-p^{-s})-\log(1-p^{-w})}{p^{-w}-p^{-s}}\Big) .\]
In what follows, the letters~$\bx$,~$\by$ and~$\bz$ shall denote complex numbers.

Whenever~$\bz\in\CpR$, taking principal determinations of the logarithms, we have
\[ \Re\Big(\frac{\bz^{1/2}-\bz^{-1/2}}{\log \bz}\Big) > 0 \]
where the fraction is analytically extended with value~$1$ at~$z=1$. It follows that the function
\begin{equation}\label{definition-g} g(\bz) := \log\Big(\frac{\bz^{1/2}-\bz^{-1/2}}{\log \bz}\Big) \end{equation}
is a well-defined analytic function of~$\bz\in\CpR$. Since we have~$(1-\bx)/(1-\by)\in\CpR$ for all~$(\bx, \by)\in\cU^2$, it follows that the function
\begin{equation}\label{definition-h} \Xi(\bx, \by) := -\frac12\log(1-\bx)-\frac12\log(1-\by) - g\Big(\frac{1-\bx}{1-\by}\Big) \end{equation}
is an analytic function of~$(\bx, \by)\in\cU^2$. When~$\bx, \by\in(-1, 1)$, we have
\[ \exp\{\Xi(\bx, \by)\} = \frac{\log(1-\bx)-\log(1-\by)}{\by-\bx}. \]
This identity therefore holds on~$\cU^2$ by analytic continuation. Putting
\begin{equation}\label{def-fy} f_y(s, w) := \sum_{p\leq y}\Xi(p^{-s}, p^{-w}) ,\end{equation}
we obtain that~$f_y(s, w)$ is an analytic function of~$(s, w)\in\cH^2$, and
\[ F_y(s, w) = \exp\{f_y(s, w)\}. \]

For any~$(k,\ell)\in\bfN^2$ and function~$f(\bx, \by)$ of class~$\cC^{k+\ell}$, we shall use the notation
\[ \partial_{k\ell}f := \frac{\partial^{k+\ell}f}{\partial \bx^k\partial \by^\ell} .\]
The hessian will play an important role: for a class~$\cC^2$ function~$f$ of two variables, we denote
\[ \hess[f] := (\partial_{20}f)(\partial_{02}f) - (\partial_{11}f)^2 .\]

In the rest of the paper,~$\Xi(\bx, \by)$ will always denote the function defined by equations~\eqref{definition-h} and~\eqref{definition-g} in the proof of the previous lemma. The next lemma regroups some useful facts concerning the power series expansion of~$\Xi(\bx, \by)$.
\begin{lemme}\label{props-h}
\begin{enumerate}[(i)]
\item For some sequence of \emph{positive} coefficients~$(d_{k,\ell})_{k+\ell\geq 1}$ with~$d_{1,0}=d_{0,1}=1/2$, the power series expansion of~$\Xi(\bx, \by)$ at~$(0,0)$ is
\begin{equation}\label{dvpt-h} \Xi(\bx, \by) = \sum_{k+\ell\geq 1}d_{k, \ell}\bx^k\by^\ell \qquad ((\bx, \by)\in\cU^2). \end{equation}
\item For some analytic function~$\xi(\bx, \by)$ of~$(\bx, \by)\in\cU^2$, we may write
\begin{equation}\label{def-xi} g\Big(\frac{1-\bx}{1-\by}\Big) = (\bx-\by)^2\xi(\bx, \by) \qquad ((\bx, \by)\in\cU^2) .\end{equation}
\item For some sequences~$(d'_{k,\ell})_{k, \ell\geq 0}$ and~$(d''_{k,\ell})_{k, \ell\geq 0}$ of positive numbers with~$d''_{0,0}=d'_{0,0}=1/24$, we have
\begin{equation}\label{dvpt-xi} \xi(\bx, \by) = \sum_{k, \ell\geq 0} d'_{k, \ell}\bx^k\by^\ell \qquad ((\bx, \by)\in\cU^2) ,\end{equation}
\begin{equation}\label{expr-ders-xi} \partial_{k\ell}\xi(\bx, \bx) = \frac{d''_{k, \ell}}{(1-\bx)^{k+\ell+2}} \qquad (\bx\in\cU) .\end{equation}
\item For all~$(\bx, \by)\in(0, 1)$, we have
\begin{equation}\label{pos-hessienne-h} [(\partial_{20}\Xi)(\partial_{02}\Xi) - (\partial_{11}\Xi)^2](\bx, \by) > 0.\end{equation}
\end{enumerate}
\end{lemme}
The useful feature in points~(i) and~(iii) is the positivity of the coefficients, which will provide a neat way to establish bounds on~$F_y(s, w)$.
\begin{proof}
Recall that the function~$g$ is defined by~\eqref{definition-g}. Note that~$g(\bz) = O(\log(|\bz|+|\bz|^{-1}))$ uniformly for~$\bz\in\CpR$, and~$g(1)=0$. Thus, whenever~$\bz\not\in\bfR_-$ and~$\Gamma$ is an oriented circle inside~$\CpR$ circling around~$\bz$ counter-clockwise, the Cauchy formula yields
\[ g(\bz) = \frac{\bz-1}{2\pi i}\oint_{\Gamma}\frac{g(w)\dd w}{(w-\bz)(w-1)} = \frac{\bz-1}{2\pi i}\int_{-\infty}^0 \frac{g(w+0i)-g(w-0i)}{(w-\bz)(w-1)}\dd w, \]
where the last equality follows from modifying the contour of integration into a Hankel contour,  first from~$-\infty$ to~$0$ with argument~$\pi$, then from~$0$ to~$-\infty$ with argument~$-\pi$.
%(note that~$g(\bz)$ is integrable in the neighborhood of~$0$ in~$\CpR$).
Setting~$t=1/(1-w)$, we obtain
\[ g(\bz) = (1-\bz)\int_0^1 \frac{\cK(t)\dd t}{1-t(1-\bz)}, \qquad\text{where } \cK(t) = \frac1\pi\arctan\Big(\frac1\pi\log\Big(\frac{t}{1-t}\Big)\Big) \quad (t\in(0, 1)) \]
which we extend by continuity at~$t=0$ and~$1$. Letting~$\bz=(1-\bx)/(1-\by)$, we deduce that
\begin{equation}\label{eq-g-integrale} g\Big(\frac{1-\bx}{1-\by}\Big) = (\bx-\by)\int_0^1\frac{\cK(t)\dd t}{1-(t\bx+(1-t)\by)} .\end{equation}
Note that the function~$\cK$ is differentiable in~$(0, 1)$ and
\[ \cK'(t) = \frac{1}{t(1-t)\big\{\pi^2 + \log(t/(1-t))^2\big\}} > 0 \qquad (0<t<1). \]
Expanding the rational fraction in the RHS of~\eqref{eq-g-integrale} as a power series, and taking into account the factor~$(\by-\bx)$, we obtain for some coefficients~$(\tilde{d}_j)_{j\geq 0}$ the expression
\begin{equation}\label{dvpt-g-pos}\begin{aligned} \ &g\Big(\frac{1-\bx}{1-\by}\Big) \\=\ & \sum_{j\geq 0}\tilde{d}_j\{\bx^j+\by^j\} +   \int_0^1 \cK(t)\sum_{k, \ell\geq 1}\bx^k\by^\ell{\ell+k\choose k}\frac1{k+\ell}\Big\{kt^{k-1}(1-t)^{\ell} - \ell t^{k}(1-t)^{\ell-1}\Big\}\dd t \\
=\ & \sum_{j\geq 0}\tilde{d}_j\{\bx^j+\by^j\} - \sum_{k, \ell\geq 1}\bx^k\by^\ell {k+\ell\choose k}\frac1{k+\ell}\int_0^1 \cK'(t)t^k(1-t)^\ell\dd t \end{aligned}\end{equation}
by an integration by parts. The point here is that the coefficients of terms~$\bx^k\by^\ell$ with positive exponents are negative. We return now to~$\Xi(\bx, \by)$. Setting~$\by=0$, we have
\[ \Xi(\bx, 0) = \log\Big(-\frac{\log(1-\bx)}{\bx}\Big) \quad (\bx\in\cU). \]
By considering the derivative of this expression, it is easily obtained that the coefficients~$(d_{j,0})_{j\geq 1}$ in the expansion~$\Xi(\bx, 0) = \sum_{j\geq 1}d_{j,0}\bx^{j}$ are positive, and~$d_{1,0}=1/2$. Using this expansion, the expression~\eqref{dvpt-g-pos} for~$g$ and equation~\eqref{definition-h} (as well as the symmetry between~$\bx$ and~$\by$), we finally get
\[ \Xi(\bx, \by) = \sum_{j\geq 1}d_{j,0}\big(\bx^{j}+\by^{j}) + \sum_{k, \ell\geq 1}\bx^k\by^\ell {k+\ell\choose k}\frac1{k+\ell}\int_0^1 \cK'(t)t^k(1-t)^\ell\dd t = \sum_{k+\ell\geq 1}d_{k, \ell}\bx^k\by^\ell \]
say, where the coefficients~$(d_{k,\ell})_{k+\ell\geq 1}$ are positive and~$d_{1,0}=d_{0,1}=1/2$. This yields~\eqref{dvpt-h}.

We continue with the expression~\eqref{eq-g-integrale}. Since~$\cK(1-t)=-\cK(t)$, we deduce
\[ g\Big(\frac{1-\bx}{1-\by}\Big) = (\bx-\by)^2\int_0^1\frac{(t-1/2)\cK(t)\dd t}{(1-(t\by+(1-t)\bx))(1-(t\bx+(1-t)\by))} \]
from which we deduce the existence of the function~$\xi(\bx, \by)$ satisfying~\eqref{def-xi} and its analyticity. Note that~$(t-1/2)\cK(t)\geq 0$ for~$t\in[0, 1]$. For all~$t\in[0, 1]$, we let
\[ R_t(\bx, \by) := \frac1{(1-(t\by+(1-t)\bx))(1-(t\bx+(1-t)\by))} = \sum_{k, \ell\geq 0}r_{k,\ell}(t)\bx^k\by^\ell, \]
for some numbers~$r_{k,\ell}(t)$, by expanding the rational fraction as a power series in~$t\bx+(1-t)\by$ and~$t\by+(1-t)\bx$, which in turn is a power series in~$\bx$ and~$\by$ whose coefficients are polynomial combinations of~$t$ and~$1-t$ with positive coefficients. Therefore, for all~$k, \ell\geq 0$, $r_{k,\ell}(t)$ is a non-zero polynomial in~$t$ with~$r_{k,\ell}(t)\geq 0\quad(t\in[0, 1])$. Setting
\[ d'_{k,\ell} := \int_0^1(t-1/2)\cK(t)r_{k,\ell}(t)\dd t, \]
the expansion~\eqref{dvpt-xi}, along with the positivity of the coefficients, follows at once. Furthermore, it is easily seen by induction that for all~$k, \ell\geq 0$,
\[ \partial_{k,\ell}R_t(\bx, \by) = \sum_{1\leq j\leq k+\ell+1}\frac{P^{(j)}_{k,\ell}(t)}{(1-(t\by+(1-t)\bx))^j(1-(t\bx+(1-t)\by))^{k+\ell+2-j}} \]
for some non-zero polynomials~$P_{k,\ell}^{(j)}(t)\geq 0\ (t\in[0, 1])$. This yields the equation~\eqref{expr-ders-xi}. The fact that~$d'_{0,0} = 1/24$ is a simple calculation; it implies that~$d''_{0,0}=1/24$ by specialization at~$\bx=0$.

The inequality~\eqref{pos-hessienne-h} is proved by a direct computation. Let~$\bz := (1-\bx)/(1-\by) > 0$. Then
\[ [(\partial_{20}h)(\partial_{02}h)-(\partial_{11}h)^2](\bx, \by) = \frac{\frac{1+\bz}{\bz-1}\log \bz - 2}{(1-\bx)^2(1-\by)^2(\log \bz)^2} \]
which is extended by continuity as~$1/(6(1-\bx)^4)$ when~$\bx=\by$. When~$\bz\neq 1$, the positivity of the numerator is easy to establish.

\end{proof}

We introduce for~$\delta\geq0$ the subset
\[ \cD_\delta(\alpha ; y) := \big\{(\sigma, \kappa)\in(0, 1]^2 : (\sigma-\alpha)(\alpha-\kappa)\geq 0, \quad \frac1{1+\delta}\leq \frac{1-2^{-\sigma}}{1-2^{-\kappa}} \leq 1+\delta \ \text{ and }\ |\sigma-\kappa|\log y \leq \delta\big\}. \]
The first condition simply means that~$\alpha$ is between is~$\sigma$ and~$\kappa$. The other guarantee that~$\sigma$ and~$\kappa$ are adequately close to each other. Note that~$\cD_{\delta'}(\alpha ; y) \subset \cD_{\delta}(\alpha ; y)$ whenever~$\delta'\leq \delta$, and that if~$(\sigma, \kappa)\in\cD_{\delta}(\alpha ; y)$, then uniformly for~$p\leq y$, we have
\begin{equation}\label{approx-cD-alpha} 1-p^{-\sigma} = (1-p^{-\alpha})\{1+O(\delta)\}, \qquad \sigma = \alpha\{1+O(\delta+(\log y)^{-1})\}, \qquad p^{\sigma} = p^\alpha\{1+O(\delta)\}, \end{equation}
and similarly for~$\kappa$.

In the next lemma, we deduce from the properties of the series~$\Xi(\bx, \by)$ some information about the function~$f_y(s, w)$ defined in~\eqref{def-fy}. Recall that~$\sigma_2$ and~$\tsigma_2$ were defined by~\eqref{def-sigmas}.

\begin{lemme}\label{lemme-props-fy}
\begin{enumerate}[(i)]
For some absolute constant~$\delta_0>0$ and all~$2\leq y\leq x$, the following assertions hold.
\item For all~$(s, w)\in\cH^2$ and~$k,\ell\geq0$, we have~$\Re\big\{\partial_{k\ell}f_y(s, w)\big\} \leq  \partial_{k\ell}f_y(\sigma, \kappa)$.
\item For all~$\sigma, \kappa > 0$, we have
\[ [\partial_{20}f_y + \partial_{11}f_y](\sigma, \kappa) > 0 \quad\text{and}\quad \hess[f_y](\sigma, \kappa) > 0. \]
\item For all non negative integers~$k, \ell$ with~$(k, \ell) \neq (0, 0)$, we have
\[ \partial_{k\ell}f_y(\sigma, \kappa) \ll_{k,\ell} |\phi_{k+\ell}(\alpha, y)| \qquad ((\sigma, \kappa)\in\cD_{\delta_0}(\alpha ; y)) .\]
\item Whenever~$0\leq\delta\leq\delta_0$ and~$(\sigma, \kappa)\in\cD_\delta(\alpha ; y)$, we have
\begin{equation*}\begin{aligned}
\partial_{20}f_y(\sigma, \kappa) + \partial_{11}f_y(\sigma, \kappa) &= \frac{\sigma_2}2 + O(\delta\sigma_2), \\ \hess[f_y](\sigma, \kappa) &= \frac{\sigma_2}4\big\{\sigma_2 - \frac{\tsigma_2}3\big\} + O(\delta\sigma_2^2).
\end{aligned}\end{equation*}
\end{enumerate}
\end{lemme}

\begin{proof}
Part~(i) follows immediately from part~(i) of Lemma~\ref{props-h}. Indeed, for all fixed indices~$k, \ell\geq 0$, we can write~$\partial_{k\ell}\Xi(\bx, \by) = \sum_{j_1, j_2\geq 0}\tilde{d}_{j_1, j_2}\bx^{j_1} \by^{j_2}$ for some non-negative coefficients~$\tilde{d}_{j_1, j_2}$ depending on~$k$ and~$\ell$. Then
\begin{align*}
\Re\big\{\partial_{k\ell}f_y(s, w)-\partial_{k\ell}f_y(\sigma, \kappa)\big\} = -\sum_{j_1, j_2\geq 0}\tilde{d}_{j_1, j_2}\sum_{p\leq y}\frac{1-\cos((j_1\tau+ j_2t)\log p)}{p^{j_1\sigma+j_2\kappa}} \leq 0 \end{align*}
by positivity. Regarding part~(ii), the inequality~$[\partial_{20}f_y + \partial_{11}f_y](\sigma, \kappa) > 0$ also follows immediately by linearity from the equality
\[ [\partial_{20}f_y + \partial_{11}f_y](\sigma, \kappa) = \sum_{p\leq y}(\log p)^2\big[\sfa\partial_{10}\Xi(\sfa, \sfb) + \sfa^2\partial_{20}\Xi(\sfa, \sfb)+\sfa\sfb\partial_{11}\Xi(\sfa, \sfb)\big]_{\substack{\sfa=p^{-\sigma}\\\sfb=p^{-\kappa}}} \]
and the positivity of the coefficients in the expansion~\eqref{dvpt-h}. Concerning the hessian, we apply the Cauchy--Schwarz inequality, getting
\begin{align*}
[(\partial_{20}f_y)(\partial_{02}f_y)](\sigma, \kappa) \geq\ & \Big(\sum_{p\leq y}(\log p)^2\sqrt{\big[(\sfa\partial_{10}\Xi(\sfa, \sfb)+\sfa^2\partial_{20}\Xi(\sfa, \sfb))(\sfb\partial_{01}\Xi(\sfa, \sfb)+\sfb^2\partial_{02}\Xi(\sfa, \sfb))\big]_{\substack{\sfa=p^{-\sigma}\\ \sfb=p^{-\kappa}}}}\Big)^2
\\\geq\ & \Big(\sum_{p\leq y}(\log p)^2\sqrt{\big[\sfa^2\sfb^2\partial_{20}\Xi(\sfa, \sfb)\partial_{02}\Xi(\sfa, \sfb)\big]_{\substack{\sfa=p^{-\sigma} \\\sfb=p^{-\kappa}}}}\Big)^2. \end{align*}
By~\eqref{pos-hessienne-h}, the last sum over~$p$ is strictly greater than
\[ \sum_{p\leq y}(\log p)^2p^{-\sigma-\kappa}\partial_{11}\Xi(p^{-\sigma}, p^{-\kappa}) = \partial_{11}f_y(\sigma, \kappa) \]
as required.

We now turn to estimating the derivatives of~$f_y$. Assume that~$(\sigma, \kappa)\in\cD_\delta(\alpha ; y)$ for some small~$\delta$. Recall that
\[ \Xi(\sfa, \sfb) = -\frac12\log(1-\sfa)-\frac12\log(1-\sfb) - (\sfa-\sfb)^2\xi(\sfa, \sfb). \]
Let~$k, \ell\geq 0$ be fixed with~$k+\ell\geq 1$. Then the derivative~$\partial_{k\ell}\Xi(\sfa, \sfb)$ can be written as a linear combination with bounded coefficients of terms assuming one of the following four shapes:
\begin{equation}\label{liste-termes-derivees-h}
\left\{\begin{aligned} &(1-\sfa)^{-k}\quad \text{if~$\ell=0$}, \qquad \text{or} \qquad (1-\sfb)^{-\ell}\quad \text{if~$k=0$,} \\
& (\sfa-\sfb)^2\partial_{k\ell}\xi(\sfa, \sfb), \\
& (\sfa-\sfb)\partial_{j_1j_2}\xi(\sfa, \sfb) \qquad \text{with~$j_1+j_2=k+\ell-1$ and~$j_i\geq 0$}, \\
& \partial_{j_1j_2}\xi(\sfa, \sfb) \qquad \text{with~$j_1+j_2=k+\ell-2$ and~$j_i\geq 0$}. \end{aligned}\right.
\end{equation}
Suppose for simplification that~$\sfa\leq \sfb$. Then for any~$j_1, j_2\geq 0$, we have
\begin{equation}
\label{comp-xi}
\partial_{j_1j_2}\xi(\sfa, \sfa)\leq \partial_{j_1j_2}\xi(\sfa, \sfb) \leq \partial_{j_1j_2}\xi(\sfb, \sfb) \ll_{j_1j_2} (1-\sfb)^{-j_1-j_2-2}
\end{equation}
by virtue of~\eqref{dvpt-xi} and~\eqref{expr-ders-xi}. Noting that~$|\sfa-\sfb|\leq 1-\sfa$, It follows that each of the four expressions given in~\eqref{liste-termes-derivees-h} is bounded by~$O_{k,\ell}((1-\sfa)^2(1-\sfb)^{-k-\ell-2})$, so that
\[ \partial_{k\ell}\Xi(\sfa, \sfb) \ll_{k, \ell} (1-\sfa)^2(1-\sfb)^{-k-\ell-2} \qquad (k+\ell\geq 1,\ \sfa\leq\sfb). \]
By symmetry, when~$\sfb\leq \sfa$ the same estimate holds if we swap~$\sfa$ and~$\sfb$ in the right-hand side.
Next, we specialize~$\sfa=p^{-\sigma}$, $\sfb=p^{-\kappa}$. By the property~\eqref{approx-cD-alpha}, if~$\delta$ is small enough, we have for all~$k, \ell\geq 0$ with~$k+\ell\geq 1$
\begin{equation}
  \label{majo-drvptl-h}
 \partial_{k\ell}\Xi(p^{-\sigma}, p^{-\kappa}) \ll_{k,\ell} (1-p^{-\alpha})^{-k-\ell}.  
\end{equation}
Differentiating the function~$(\sigma, \kappa)\mapsto \Xi(p^{-\sigma}, p^{-\kappa})$, $k$ times with respect to~$\sigma$ and~$\ell$ times with respect to~$\kappa$ yields a linear combination of terms of the shape
\[ (\log p)^{k+\ell}p^{-j_1\sigma - j_2\kappa}\partial_{j_1j_2}\Xi(p^{-\sigma}, p^{-\kappa}) \qquad (1\leq j_1+j_2\leq k+\ell)\]
each of which is bounded by~$O((\log p)^{k+\ell}(p^{\alpha}-1)^{-j_1-j_2})$ (here we used~\eqref{approx-cD-alpha} and~\eqref{majo-drvptl-h}). Summing over~$p\leq y$, we obtain
\[ \partial_{k\ell}f_y(\sigma, \kappa) \ll_{k,\ell} \sum_{p\leq y} (\log p)^{k+\ell}\Big\{\frac1{(p^\alpha-1)^{k+\ell}}+\frac1{p^\alpha-1}\Big\} \leq \tphi_{k+\ell}(\alpha, y) + u(\log y)^{k+\ell}. \]
Each of the last two terms is bounded from above by~$O(|\phi_{k+\ell}(\alpha, y)|)=O((u\log y)^{k+\ell}(\ubar)^{1-k-\ell})$ and this proves part~(iii).

We now estimate the hessian. A direct calculation reveals that
\[ \partial_{20}f_y(\sigma, \kappa) = \frac12\phi_2(\sigma, y) - \sum_{p\leq y}(\log p)^2\big[2\sfa(\sfa-\sfb)\xi + \sfa(\sfa-\sfb)^2\partial_{10}\xi + 2\sfa^2\xi + 4\sfa^2(\sfa-\sfb)\partial_{10}\xi + \sfa^2(\sfa-\sfb)^2\partial_{20}\xi\big]_{\substack{a=p^{-\sigma}\\b=p^{-\kappa}}}, \]
\[ \partial_{11}f_y(\sigma, \kappa) = -\sum_{p\leq y}(\log p)^2\big[\sfa\sfb\big\{-2\xi + 2(\sfa-\sfb)(\partial_{10}\xi - \partial_{01}\xi) + (\sfa-\sfb)^2\partial_{11}\xi\big\}\big]_{\substack{a=p^{-\sigma}\\b=p^{-\kappa}}}, \]
where we abbreviated for simplicity~$\partial_{k\ell}\xi = \partial_{k\ell}\xi(\sfa, \sfb)$. 
Using~\eqref{approx-cD-alpha} and the properties~\eqref{expr-ders-xi} and~\eqref{comp-xi}, we obtain
\[ \partial_{20}f_y(\sigma, \kappa) = \frac12\phi_2(\alpha, y) - 2\sum_{p\leq y}(\log p)^2p^{-2\alpha}\xi(p^{-\sigma}, p^{-\kappa}) + O(\delta\phi_2(\alpha, y)) ,\]
\[ \partial_{11}f_y(\sigma, \kappa) = 2\sum_{p\leq y}(\log p)^2p^{-2\alpha}\xi(p^{-\sigma}, p^{-\kappa}) + O(\delta\phi_2(\alpha, y)). \]
Using once more the properties~\eqref{dvpt-xi} and~\eqref{expr-ders-xi}, along with the value~$d''_{0,0}=1/24$, we get
\[ \partial_{20}f_y(\sigma, \kappa) = \frac{\sigma_2}2 - \frac{\tsigma_2}{12} + O(\delta\sigma_2), \qquad \partial_{11}f_y(\sigma, \kappa) = \frac{\tsigma_2}{12} + O(\delta\sigma_2) .\]
Using the symmetry of~$f_y$ with respect to~$\sigma\leftrightarrow \kappa$, we finally obtain
\begin{align*} \hess[f_y](\sigma, \kappa) = \Big(\frac{\sigma_2}2 - \frac{\tsigma_2}{12}\Big)^2 - \Big(\frac{\tsigma_2}{12}\Big)^2 + O(\delta\sigma_2^2) = \frac{\sigma_2}4\Big\{\sigma_2 - \frac{\tsigma_2}3\big\} + O(\delta\sigma_2^2) \end{align*}
which gives part~(iv) of the lemma.
\end{proof}

\subsection{Decay estimates along vertical lines}

For the saddle-point method to succeed, it is required that the tails of the integrals in~\eqref{perron2} contribute a negligible quantity. The following lemma, which provides sufficient information for this purpose, states that the decay of~$F_y(s, w)$ away from the to-be saddle-points is reasonnably good compared with what a Taylor formula at order~$2$ would predict, even in a range where the Taylor formula turns out not to be relevant. It is an analog of~\cite[Lemma~8]{HT86}. We recall our notation that~$c_1, c_2, \dots$ denote constants, which are absolute unless otherwise specified.
\begin{lemme}\label{decr-fy}
\begin{enumerate}[(i)]
\item Whenever
\[ |\sigma-\kappa|\leq c_1\sigma, \quad \max\{\sigma+|\tau|, \kappa+|t|\}\leq c_1/(\log y), \]
we have
$$ \Re\big\{f_y(s, w) - f_y(\sigma, \kappa)\big\} \leq -c_1\frac{y}{\log y}\Big\{\log\Big(1+\Big(\frac{\tau}{\sigma}\Big)^2\Big) + \log\Big(1+\Big(\frac{t}{\kappa}\Big)^2\Big) \Big\}. $$
% \vspace{.5em}

\item When~$\sigma\leq c_1/\log y$ and~$|\tau|\leq \e^{(\log y)^{3/2-\ee}}$, we have
$$ \Re\big\{f_y(s, w) - f_y(\sigma, \kappa)\big\} \leq -c_1\frac{y}{\log y}\big\{\log\big(1+\big(\frac{\tau\log x}y\big)^2\big) + O(1)\big\}. $$
% \vspace{.3em}

\item For all~$\ee>0$, there exists~$c_2(\ee)>0$ depending only on~$\ee$ such that whenever
\[ \min\{\sigma, \kappa\}\geq\ee/\log y, \quad \max\{\sigma, \kappa\}\leq 0.6, \quad \max\{|\tau|, |t|\}\leq c_2(\ee)/\log y, \]
we have
$$ \Re\big\{f_y(s, w) - f_y(\sigma, \kappa)\big\} \leq -c_2(\ee)\Big\{ \tau^2\phi_2(\sigma, y) + t^2\phi_2(\kappa, y) \Big\}. $$

\end{enumerate}
\end{lemme}
Although it is elementary, the proof of this lemma is somewhat lengthy and otherwise unrelated to the rest of the argument: it is postponed to the appendix. We deduce the following estimate for~$F_y(s, w)$.

\begin{coro}\label{decr-F}
Let~$|\tau|, |t|\leq \exp\{(\log y)^{4/3}\}$. For some absolute constants~$\delta, c_3>0$, whenever~$(\sigma, \kappa)\in\cD_\delta(\alpha ; y)$, the following holds.
\begin{enumerate}[(i)]
\item For~$\max\{|\tau|, |t|\}\gg 1/\log y$, we have
\begin{equation}\label{majo-F-haut}
\Big|\frac{F_y(s, w)}{F_y(\sigma, \kappa)}\Big| \ll \exp\Big\{-c_3\ubar\Big(\frac{\tau^2}{(1-\alpha)^2+\tau^2}+\frac{t^2}{(1-\alpha)^2+t^2}\Big) \Big\}.
\end{equation}
\item For~$\max\{|\tau|, |t|\}\leq c_3/\log y$, we have
\begin{equation}\label{majo-F-bas}
\Big|\frac{F_y(s, w)}{F_y(\sigma, \kappa)}\Big| \leq \Big\{\Big(1+\frac{\tau^2\phi_2(\alpha, y)}{y/\log y}\Big)\Big( 1+\frac{t^2\phi_2(\alpha, y)}{y/\log y}\Big)\Big\}^{-c_3\frac{y}{\log y}}.
\end{equation}
\item For~$\alpha \leq c_3/\log y$ and~$|\tau|\leq \e^{(\log y)^{3/2-\ee}}$, we have
\begin{equation}
\label{eq:majo-F-mid}
\Big|\frac{F_y(s, w)}{F_y(\sigma, \kappa)}\Big| \leq \e^{O(y/\log y)}\big(1 + \frac{(\tau\log x)^2}{(1+\tau^2)y^2}\big)^{-c_4(\ee) y/\log y}
\end{equation}
with~$c_4(\ee)>0$ depending only on~$\ee>0$.
\end{enumerate}
\end{coro}

\begin{proof}
First suppose~$\alpha\geq 0.55$. Then if~$\delta$ is sufficiently small, we have~$\sigma, \kappa\geq 0.54$. On the other hand, from~\eqref{def-xi} we see that~$\Xi(p^{-s}, p^{-w}) = -\frac12\log(1-p^{-s})-\frac12\log(1-p^{-w}) + O(p^{-2\sigma}+p^{-2\kappa})$, from which it follows that~$F_y(s, w)=\zeta(s, y)^{1/2}\zeta(w, y)^{1/2}\exp\{O(1)\}$. In this case, Corollary~\ref{decr-F} is a direct consequence of Lemma~8 of~\cite{HT86}.

Next let~$c_1$ be the constant in Lemma~\ref{decr-fy}.(ii), and suppose that~$|\tau|, |t|\leq c_1/(2\log y)$ and~$\alpha\leq c_1/(4\log y)$. If~$\delta$ is sufficiently small, this implies~$\sigma, \kappa\leq c_1/(2\log y)$. If moreover~$\delta$ is sufficiently small in terms of~$c_1$, then the conditions of Lemma~\ref{decr-fy}.(i) are fulfilled and for some constant~$c>0$, we have
\[ |F_y(s, w)| \leq F_y(\sigma, \kappa)\exp\Big\{ -c\frac{y}{\log y}\Big(\log\Big(1+\Big(\frac\tau\sigma\Big)^2\Big)+\log\Big(1+\Big(\frac t\kappa\Big)^2\Big)\Big)\Big\}. \]
Since under our hypotheses~$\sigma^2\asymp \alpha^2 \asymp y/(\phi_2(\alpha, y)\log y)$, and similarly for~$\kappa$, we have the required estimate

Assume next that~$\eta/(4\log y)<\alpha<0.55$, and~$|\tau|, |t|\leq c_1/(2\log y)$. In this case, assuming~$\delta$ is small enough, we deduce~$c_1/(5\log y)\leq \sigma\leq 0.6$ and similarly for~$\kappa$, so that the conditions of Lemma~\ref{decr-fy}.(iii) (with~$\ee<c_1/5$ being absolute) are satisfied: for some absolute constant~$c>0$, we have
\[ |F_y(s, w)| \leq F_y(\sigma, \kappa)\exp\Big\{-c\Big(\tau^2\phi_2(\sigma, y)+t^2\phi_2(\kappa, y)\Big)\Big\}. \]
Note that~$\phi_2(\sigma, y)\asymp\phi_2(\alpha, y)$ and similarly for~$\kappa$. Furthermore, we have under the current hypotheses~$\tau^2\phi_2(\alpha, y)\log y/y \ll u(\log y)/y \ll 1$. Therefore
\[ \frac{y}{\log y}\log\Big(1+\frac{\tau^2\phi_2(\alpha, y)}{y/(\log y)}\Big) \asymp \tau^2\phi_2(\alpha, y) \]
and similarly for~$t$. This yields the required estimate

Suppose next that~$\alpha<0.55$,~$|t|\leq |\tau|$ and~$|\tau|\gg 1/\log y$ (as we may without loss of generality). Then from part~(i) of Lemma~\ref{props-h} we deduce
\[ \Re\big\{f_y(s, w) - f_y(\sigma, \kappa)\big\} = -\sum_{k+\ell\geq 1}d_{k\ell}\sum_{p\leq y}\frac{1-\cos((k\tau+\ell t)\log p)}{p^{k\sigma+\ell\kappa}} \leq -\sum_{p\leq y}\frac{1-\cos(\tau\log p)}{2p^\sigma} \]
dropping all but one term by positivity. Note that~$p^\sigma \asymp p^\alpha$. It follows from Lemma~8 of \cite{HT86} that for some~$c>0$, we have
\[ \sum_{p\leq y}\frac{1-\cos(\tau\log p)}{p^\alpha} \gg \frac{c\ubar\tau^2}{(1-\alpha)^2+\tau^2} .\]
(Note that the condition~$|\tau|\geq 1/(\log y)$ in the statement of~\cite{HT86} may be relaxed to~$|\tau|\gg 1/\log y$ without changing the proof). Since the fraction in the right-hand side is an increasing function of~$|\tau|$ and~$|\tau|\geq |t|$, we obtain the required result.

Finally, if~$\alpha\leq c_3/\log y$ and~$|\tau|\leq \e^{(\log y)^{3/2-\ee}}$, exponentiating the upper bound of Lemma~\ref{decr-fy}.(iii) immediately yields the desired result.
\end{proof}

We shall use the first estimate of the previous corollary in the form of the following bounds.
\begin{coro}\label{decr-F-int}
Suppose~$(\sigma, \kappa)\in\cD_{\delta}(\alpha ; y)$ for some sufficiently small~$\delta\geq 0$. Then the following assertions hold.
\begin{enumerate}[(i)]
\item For all~$1\leq X\leq \exp\{(\log y)^{4/3}\}$ and $\lambda \in\bfR$,
\[ \int_{-X}^{X}|F_y(\sigma + i\lambda\tau, \kappa+i\tau)|\dd\tau \ll F_y(\sigma, \kappa)\big\{ 1 + X\e^{-c_5\ubar}\big\}. \]
(By symmetry the same bound holds for the analogous integral over~$t$).
\item For all~$0<\delta\leq 1$, $(\mu_1, \mu_2)\in(0, 2]^2$, $(\lambda_1, \lambda_2)\in\bfR^2$ and~$5\leq X\leq \exp\{(\log y)^{4/3}\}$, we have
\begin{equation}
\begin{aligned}
\iint_{(\ast)} |F_y(\sigma+&i(\tau+\lambda_2t), \kappa+i(\tau+\lambda_1t))|\frac{\dd\tau}{\mu_1+|\tau|}\frac{\dd t}{\mu_2+|t|} \\
&\ \ll_\delta F_y(\sigma, \kappa)\log(X/(\mu_1\mu_2))^2 H(\ubar)^{-c_5\delta^2}(\log x)^{-1},
\end{aligned}\label{majo-F-int-2}
\end{equation}
\[ \text{where the integration domain is }(\ast) = \left\{(\tau, t) : \begin{array}{l}\max\{|\tau|, |t|\}\leq X, \\ \max\{|\tau+\lambda_1t|, |\tau+\lambda_2t|\}\geq \delta/\log y \end{array} \right\}.\]
\end{enumerate}
\end{coro}

\begin{proof}
The proof is very similar to the calculations of~\cite[pages~277 and~279]{HT86}. We only sketch the details. 

Part~(i) follows in a straightforward way from bounding trivially by~$O(F_y(\sigma, \kappa))$ the contribution to the integral of~$|\tau|\leq 1$, and bounding by~$O(X\e^{-c\ubar}F_y(\sigma, \kappa))$, for some~$c>0$, the contribution of~$|\tau|\geq 1$ using Corollary~\ref{decr-F}.(i).

Regarding part~(ii), first note that~$\log x = H(\ubar)^{o(1)}$ as~$\ubar\to\infty$ as soon as the two conditions~$y\geq (\log\log x)^{1+\ee}$ and~$u\leq (\log\log y)^{1+\ee}$. The second was assumed from the outset of our argument. Let us first assume that~$y> (\log \log x)^2$. Then Corollary~\ref{decr-F}.(ii) yields
\begin{align*}
&\iint_{(\ast)} |F_y(\sigma+i(\tau+\lambda_1t), \kappa+i(\tau+\lambda_2t))|\frac{\dd\tau}{\mu_1+|\tau|}\frac{\dd t}{\mu_2+|t|} \\ \ll\ &F_y(\sigma, \kappa)\Bigg(\sup_{\nu\in\bfR}\int_{\delta/\log y}^X\exp\big\{-c\ubar\frac{\tau^2}{(1-\alpha)^2+\tau^2}\big\}\frac{\dd \tau}{\min\{\mu_1, \mu_2\}+|\tau-\nu|}\Bigg)^2
\end{align*}
for some~$c>0$. Since for any~$|\tau|\geq \delta/\log y$, we have
\[ \frac{\tau^2}{(1-\alpha)^2+\tau^2}\geq \frac{\delta^2}{\delta^2+((1-\alpha)\log y)^2} \gg \frac{\delta^2}{(\log\ubar)^2}, \]
we obtain the bound
\[ \ll F_y(\sigma, \kappa)H(\ubar)^{-c'\delta^2}\log(X/(\mu_1\mu_2))^2 \]
for some~$c'>0$. Since~$(\log x) \ll_{\delta,\ee} H(\ubar)^{c'\delta^2/2}$, the above is an acceptable error term.

If on the contrary~$y\leq (\log\log x)^2$, then by Corollary~\ref{decr-F}.(iii), we have
\begin{align*}
&\iint_{(\ast)} |F_y(\sigma+i(\tau+\lambda_1t), \kappa+i(\tau+\lambda_2t))|\frac{\dd\tau}{\mu_1+|\tau|}\frac{\dd t}{\mu_2+|t|} \\ \ll\ & \e^{O(y/\log y)}F_y(\sigma, \kappa)\Bigg(\sup_{\nu\in\bfR}\int_{\delta/\log y}^X
\big(1 + \frac{(\tau\log x)^2}{(1+\tau^2)y^2}\big)^{-c y/\log y}\frac{\dd \tau}{\min\{\mu_1, \mu_2\}+|\tau-\nu|}\Bigg)^2
\end{align*}
for some~$c>0$. We certainly have, for~$\tau\geq \delta/\log y$,
$$ \big(1 + \frac{(\tau\log x)^2}{(1+\tau^2)y^2}\big) \gg \delta^2 \log(u/y) \asymp \delta^2 \log x, $$
so that for some~$c'>0$, we have a bound
$$ \ll_\delta F_y(\sigma, \kappa) (\log x)^{-c' \delta^2 y/\log y} \log(X/(\mu_1\mu_2))^2. $$
This is clearly acceptable since~$y\leq (\log\log x)^2$.
\end{proof}

\subsection{The saddle-points}

Let~$2\leq y\leq x$, and~$\gamma =  v\cu$ with~$v\in\bfR$, $|v|\leq(\log x)/\cu$. We are interested in the properties of the pair of positive absciss\ae{} satisfying
\[ (\beta_1, \beta_2) = \underset{(\sigma, \kappa)\in(\bfR_+^*)^2}{\argmin\ }\big( x^{(\sigma + \kappa)/2}\e^{\gamma(w-s)}F_y(\sigma, \kappa)\big). \]
This pair will be more easily dealt with if defined by extrapolation from the case~$v=0$. We let
\[ \beta : (-v_0, v_0) \to \bfR \]
be the maximal solutions (here~$v_0\in\bfR_+^*\cup\{\infty\}$) to the differential equation
\begin{equation}\label{def-betas}
\beta'(v) = \cu\frac{[\partial_{02}f_y+ \partial_{11}f_y](\beta(v), \beta(-v))}{\hess[f_y](\beta(v), \beta(-v))} \end{equation}
satisfying the initial condition
\[ \beta(0) = \alpha(x, y). \]
\begin{lemme}\label{props-beta}
For all~$|v|<v_0$, the couple~$(\beta(v), \beta(-v))$ satisfies the saddle-point equation
\begin{equation}\label{systeme-pt-col}
\left\{\begin{aligned} \partial_{10}f_y(\beta(v), \beta(-v)) + \frac{\log x}2 - v\cu = 0, \\ \partial_{01}f_y(\beta(v), \beta(-v)) + \frac{\log x}2 + v\cu = 0. \end{aligned}\right. \end{equation}
Moreover, the function~$\beta$ is defined in the interval
\begin{equation}\label{def-cV}
\cV = \cV(x, y) = [-c_6\sqrt{\ubar}, c_6\sqrt{\ubar}] ,
\end{equation}
and for~$v\geq0, v\in\cV$, we have
\[ (\beta(v), \beta(-v))\in\cD_{\delta_0}(\alpha ; y), \qquad 0\leq \beta(v)-\beta(-v) \asymp \frac{v}{\sqrt{\sigma_2}} \]
where~$\delta_0>0$ is an absolute constant such that Lemma~\ref{lemme-props-fy}.(iii)-(iv) hold.
\end{lemme}

\begin{proof}
That~$(\beta(v), \beta(-v))$ satisfies the saddle-point equation can be seen by differentiating the system~\eqref{systeme-pt-col} with respect to~$v$, granted that it is satisfied at~$v=0$. To check this last fact, we remark that from the definitions~\eqref{def-fy}, \eqref{definition-h} and Lemma~\ref{props-h}.(ii) (more specifically, using the fact that~$g((1-\bx)/(1-\by))$ vanishes at order~$2$ when~$\bx=\by$), we have
\[ \partial_{10}f_y(\sigma, \sigma) = \partial_{01} f_y(\sigma, \sigma) = \frac12\phi_1(\sigma, y) \qquad (\sigma>0).\]
Then, when~$v=0$, both equations in~\eqref{systeme-pt-col} reduce to the definition of~$\alpha(x, y)$.

From Lemma~\ref{lemme-props-fy}.(ii), we have that~$\beta'(v)>0$ for~$|v|<v_0$. Let~$\delta_0>0$ be, as in the statement, an absolute constant such that Lemma~\ref{lemme-props-fy}.(iii)-(iv) hold, and let
\begin{equation}\label{def-vm}
v_m := \inf\big\{v\in(0, v_0) : \ \beta(v)-\beta(-v) = \frac{c_{7}\ubar}{u\log y}\text{ or } \frac{2^{-\beta(-v)}-2^{-\beta(v)}}{1-2^{-\beta(-v)}} = \delta_0 \big\} > 0
\end{equation}
where~$c_{7}>0$ is absolute and such that~$0<c_{7}\leq\delta_0$ and~$\alpha\geq 2c_{7}\ubar/(u\log y)$. Since~$\ubar\leq u$, we have
\[ (\beta(v), \beta(-v)) \in \cD_{\delta_0}(\alpha ; y), \qquad (0\leq v<v_m) .\]
In particular,~$\beta_1, \beta_2$ and their derivatives have a limit at~$v=v_m$ and the theorem of Cauchy--Lipschitz yields~$v_m<v_0$. Lemma~\ref{lemme-props-fy}.(iv) ensures that for~$0\leq v\leq v_m$ we have~$\beta'(v)\asymp (\log x)/(\cu\sigma_2) \asymp 1/\sqrt{\sigma_2}$. It follows that
\begin{equation}
\beta(v)-\beta(-v) \asymp \frac{v}{\sqrt{\sigma_2}} \quad (0\leq v\leq v_m).\label{estim-diff-beta}
\end{equation}
We claim that~$v_m\gg \sqrt{\ubar}$. Indeed, assume first that the limiting condition in~\eqref{def-vm} is~$\beta(v_m)-\beta(-v_m) = c_{7}\ubar/(u\log y)$. Then from~\eqref{estim-diff-beta} it follows that~$v_m/\sqrt{\ubar}\gg 1$. If on the contrary the limiting condition in~\eqref{def-vm} is~$(2^{-\beta(-v_m)}-2^{-\beta(v_m)})/(1-2^{-\beta(-v_m)}) = \delta_0$, then we write this condition as~$f(v_m)=0$, where~$f(v) := \delta_0 + 2^{-\beta(v)} - (1+\delta_0)2^{-\beta(-v)} \quad (0\leq v\leq v_m)$.
We have
\[ f(0) = \delta_0(1-2^{-\alpha}) \gg \alpha, \qquad f'(v) \ll |\beta'(v)| \ll 1/\sqrt{\sigma_2} \quad (0\leq v\leq v_m) \]
from which we deduce that~$v_m \gg \alpha \sqrt{\sigma_2} \gg \sqrt{\ubar}$. This proves the lemma.
\end{proof}

The next lemma describes more precisely the variations of~$\beta$ and of the quantities~$x^{(\beta_1+\beta_2)/2}\e^{\gamma(\beta_2-\beta_1)}F_y(\beta_1,\beta_2)$ and~$\hess[f_y](\beta_1, \beta_2)$ (where~$(\beta_1, \beta_2) = (\beta(v), \beta(-v))$) with respect to~$v$.
\begin{lemme}\label{lemme-E}
Define for all~$2\leq y\leq x$ and~$v\in\cV$
\begin{equation}\label{def-E}
E(v) = E(v ; x, y) := \log(x^{(\beta_1+\beta_2)/2}\e^{\gamma(\beta_2-\beta_1)}F_y(\beta_1,\beta_2))  \qquad ((\beta_1, \beta_2) = (\beta(v), \beta(-v)).
\end{equation}
Then for some sequence of functions~$(b_j(x, y))_{j\geq 0}$ satisfying
$$ b_0(x, y) = -1/2, \qquad b_j(x, y) \ll_j (\ubar)^{-j}, $$
and for each fixed~$k\geq 1$, we have the Taylor expansion
\begin{equation}
\label{taylor-E}
E(v) = E(0) + \sum_{j=0}^{k-1} b_j(x, y) v^{2j+2} + O\Big(\frac{v^{2k+2}}{(\ubar)^k}\Big) \qquad (v\in\cV).
\end{equation}
Moreover, we have
\begin{equation}
\label{taylor-hess}
\hess[f_y](\beta_1(v), \beta_2(v)) = \hess[f_y](\alpha, \alpha)\Big\{1 + O\Big(\frac{v^2}{\ubar}\Big)\Big\},
\end{equation}
\begin{equation}
\label{taylor-beta}
\beta_1(v) = \alpha - \frac{v}{\cu} + O\Big(\frac{v^2}{u\log y}\Big).
\end{equation}
\end{lemme}
\begin{proof}
First we note that
\[ E(v) = f_y(\beta(v), \beta(-v)) + \frac{\beta(v) + \beta(-v)}2\log x - v(\beta(v)-\beta(-v)) \cu. \]
By the saddle-point property~\eqref{systeme-pt-col}, we have
\[ E'(v) = -\cu\big(\beta(v)-\beta(-v)\big) \qquad (v\in\cV). \]
We wish to differentiate this expression further. In order to simplify the presentation, we introduce the following temporary notation. For~$m\geq 2$, let~$\bfD_m$ be the set of linear combinations with coefficients independent of~$x$ and~$y$ of functions of the shape~$v\mapsto \partial_{k\ell}f_y(\beta_1(v), \beta_2(v))$ defined for~$v\in\cV$, where~$k+\ell=m$. We also denote by~$\bfD_{m_1}\cdots\bfD_{m_k}$ the set of products~$f_1\cdots f_k$ where for each~$j$, ~$f_j\in\bfD_{m_j}$~; and we write~$\bfD_m^r = \bfD_m\cdots\bfD_m$ ($r$ times). Using the shorthand
\[ \frakH : v\mapsto \hess[f_y](\beta(v), \beta(-v)), \quad \frakH \in\bfD_2^2 \]
we have from~\eqref{def-betas}
\begin{equation}
  \label{betap-D2}
  \beta' \in \Big( \frac{\cu}{\frakH}\Big)\cdot\bfD_2 .
\end{equation}
(which reads ``the function~$v\mapsto (\frakH(v)/\cu)\beta'(v)$ is in~$\bfD_2$'', \textit{etc.}) It follows that
\[ E'' \in \Big(\frac{\cu^2}{\frakH}\Big)\cdot\bfD_2. \]
By differentiating further with respect to~$v$, we obtain
\[ E''' \in \Big(\frac{\cu^3}{\frakH^3}\Big)\cdot\big(\bfD_2^3\bfD_3\big), \qquad E^{(4)} \in \Big(\frac{\cu^4}{\frakH^5}\Big)\cdot\big(\bfD_2^6\bfD_4 + \bfD_2^5\bfD_3^2\big). \]
More generally, an induction over~$j$ readily yields
$$ E^{(j)} \in \Big(\frac{\cu^j}{\frakH^{2j-3}}\Big)\cdot\big(\sum_{\substack{\sum m r_m = 7j-12 \\ \sum r_m = 3j-5}} \prod_{m} \bfD_m^{r_m}\big) \qquad (j\geq 2) $$
where the summation is over sequences of non-negative integers~$(r_m)_{m\geq 2}$ satisfying
$$ \sum_{m\geq 2} m r_m = 7j-12, \qquad \sum_{m\geq 2} r_m = 3j-5. $$
Recall that~$\sigma_m := |\phi_m(\alpha, y)|$. The definition of~$\cu$ and Lemma~\ref{lemme-props-fy} imply that for~$v\in\cV$,
\[ \cu \ll \sigma_2^2, \qquad \frakH(v) \asymp \sigma_2^2, \]
\[ \|f\|_{\infty}\ll_m |\sigma_m| \ll_m (u\log y)^m(\ubar)^{1-m} \qquad (m\geq 2,\quad f\in\bfD_m). \]
It follows that for all~$j\geq 2$,
\[ E^{(j)}(v)\ll_j \sigma_2^{6-7j/2} \sum_{\substack{\sum m r_m = 7j-12 \\ \sum r_m = 3j-5}} \prod_{m} \sigma_m^{r_m} \ll_j (\ubar)^{1-j/2}. \]
Since the function~$E$ is even, the estimate~\eqref{taylor-E} and the bound~$b_j(x, y)\ll (\ubar)^{-j}$ are a consequence of the Taylor formula, and there remains to compute~$b_0(x, y)$. Lemma~\ref{lemme-props-fy}.(iv) applied with the parameters~$(\sigma, \kappa)=(\alpha, \alpha)\in\cD_0(\alpha ; y)$ yields
\[ E''(0) = -2\cu\beta'(0) = -\frac{4\cu^2}{\sigma_2 - \tsigma_2/3} = -1 \]
by definition of~$\cu$. This proves that~$b_0(x, y)=-1/2$.

Estimate~\eqref{taylor-hess} follows on the same lines. Indeed, since~$\frakH\in\bfD_2^2$, we have
\[ \frakH' \in \Big(\frac{\cu}{\frakH}\Big)\cdot\big(\bfD_2^2\bfD_3\big), \qquad \frakH''\in\Big(\frac{\cu^2}{\frakH^3}\Big)\cdot\big(\bfD_2^5\bfD_4+\bfD_2^4\bfD_3^2\big). \]
We deduce~$\|\frakH''\|_\infty \ll \sigma_4+\sigma_3^2\sigma_2^{-1} \asymp \frakH(0)/\ubar$. Since~$\frakH$ is even,~$\frakH'(0) = 0$ and the estimate~\eqref{taylor-hess} follows from a Taylor formula at order~$2$.

Finally, from~\eqref{betap-D2} we obtain
\begin{equation}\label{betas-D3}
\beta'' \in \Big( \frac{\cu^2}{\frakH^3}\Big)\cdot\big(\bfD_2^3\bfD_3\big)  \end{equation}
so that a Taylor formula at order~$2$ yields
\[ \beta(v)  = \alpha + \frac{v}{\cu} + O\Big(\frac{v^2\sigma_3}{\sigma_2^2}\Big) =  \alpha + \frac{v}{\cu} + O\Big(\frac{v^2}{u\log y}\Big) \]
as claimed.
\end{proof}

\section{Proof of Theorem~\ref{thm-pcp}}

Let~$2\leq y\leq x$ be large numbers,~$v\geq0$ such that~$v\in\cV$ (which we recall was defined in~\eqref{def-cV}), and
\[ \gamma := v\cu \in [-(\log x)/2, (\log x)/2] \]
if the constant~$c_6$ in the definition of~$\cV$ was chosen small enough. Recall that~$\beta = \beta(v)$ is defined by~\eqref{def-betas}. By the definition~\eqref{def-D}, swapping the sums over~$n$ and~$d$, we recall that
\[ D(x, y ; \gamma) = \frac{S(x, y ; \gamma)}{\Psi(x, y)}, \]
where
\[ S(x, y ; \gamma) := \underset{\substack{P(nd)\leq y \\ n\e^{2\gamma} \leq d\leq x/n}}{\sum\sum} \frac1{\tau(nd)} .\]
Let
$$ \cR(v) = \cR(x, y ; v) := E(v) - E(0),  \qquad (v\in\cV), $$
where~$E(v)$ is the quantity defined in~\eqref{def-E}.
\begin{prop}\label{prop:key-prop}
Let~$v_m \in \cV$, $v_m\geq 1$ be fixed. Assume~$2\leq y\leq x$ and~$0\leq v\leq v_m$. Then we have
\begin{equation}
\frac{S(x, y ; \gamma)}{\Psi(x, y)} = \int_v^{v_m} \Big\{1 + O\Big(\frac{1+z^2}{\ubar}\Big)\Big\}\e^{\cR(z)}\frac{\dd z}{\sqrt{2\pi}} + O\Big(\frac{\e^{\cR(v_m)}}{v_m} + \frac{\e^{\cR(v)}}{\ubar}\Big). \label{eq:S-objectif}
\end{equation}
\end{prop}
\begin{proof}[Proof that Proposition~\ref{prop:key-prop} implies Theorems~\ref{thm-pcp} and~\ref{thm:pcp-2}]
Let~$v_m = v_{\text{max}} \asymp (\ubar)^{k/(2k+2)}$ be given. Recall that~$\cR_k(v)$ is defined by~\eqref{eq:def-Rk}. By Lemma~\ref{lemme-E}, we have
$$ \cR(z) = \cR_k(z) + O\Big(\frac{z^{2k+2}}{(\ubar)^k}\Big) \qquad (0\leq z \leq v_m). $$
The error term is absolutely bounded. Let
$$ I(v) := \int_v^{2v_m} \e^{\cR_k(z)}\frac{\dd z}{\sqrt{2\pi}}, \qquad (0\leq v\leq v_m). $$
Then it is easily verified that for~$0\leq v\leq v_m$,
$$ \e^{\cR(v)} \asymp (1+v)I(v), \qquad I(v_m) \ll \frac{1+v^{2k+2}}{(\ubar)^k}I(v), $$
$$ \int_v^{v_m} z^{\ell}\e^{\cR_k(z)}\dd z \ll (1 + v^{\ell})I(v) \qquad (\ell\in\{2, 2k+2\}). $$
We deduce that
$$ \int_v^{v_m}\Big\{1 + O\Big(\frac{1+z^2}{\ubar}\Big)\Big\} \e^{\cR(z)} \frac{\dd z}{\sqrt{2\pi}} =
\Big\{ 1 + O\Big(\frac{1+v^2}{\ubar} + \frac{v^{2k+2}}{(\ubar)^k}\Big)\Big\}I(v), $$
$$ \frac{\e^{\cR(v_m)}}{v_m} \ll \frac{1+v^{2k+2}}{(\ubar)^k}I(v). $$
This implies Theorem~\ref{thm:pcp-2}. Theorem~\ref{thm-pcp} follows by specialization at~$k=1$.
\end{proof}
We define
$$ \hT := \min\{H(u), \exp((\log y)^{5/4})\}. $$
Note that~$\alpha\sqrt{\sigma_2} \ll \hT^{o(1)}$ as~$\ubar\to\infty$.

\subsubsection*{Small values of~$v$}

Let
$$ v_1 := \frac1{(\log x)\ubar\alpha\sqrt{\sigma_2}} $$
and consider first the case when~$0\leq v\leq v_1$. Note that~$\log(1/v_1)\ll\log\log x$. The right-hand side of~\eqref{eq:S-objectif} varies by an amount at most~$O(\Psi(x, y)/\ubar)$. The left-hand side of~\eqref{eq:S-objectif} varies by at most
$$ \underset{\substack{P(nd)\leq y \\ nd\leq x \\ 1 \leq \tfrac dn \leq x^{2v_1}}}{\sum\sum}\frac1{\tau(nd)} $$
where we used the rough bound~$\cu\ll (\log x)^2$. By Rankin's trick and~\eqref{majo-perron-K} for~$z=d/n$ at the height~$T=(\ubar\alpha\sqrt{\sigma_2})^2\leq \exp((\log y)^{4/3})$ (for~$y$ large enough), the quantity above is bounded by
\begin{align*}
 \ll \ & \sum_{P(nd)\leq y}\frac1{\tau(nd)}\Big(\frac{x}{nd}\Big)^\alpha\int_{-1}^1 \Big(\frac{d}{n}\Big)^{i\tau/(v_1\log x)} K(\tau)\dd\tau \\
\leq\ & \frac{x^\alpha}{\ubar\alpha\sqrt{\sigma_2}}\int_{-\ubar\alpha\sqrt{\sigma_2}}^{\ubar\alpha\sqrt{\sigma_2}} |F_y(\alpha+i\tau, \alpha-i\tau)|\dd\tau.
\end{align*}
By~\eqref{estim-psi-col}, Corollary~\ref{decr-F-int}.(i) and the fact that~$F_y(\alpha,\alpha)=\zeta(\alpha, y)$, this is
\[ \ll \Psi(x, y)\alpha\sqrt{\sigma_2}\big\{\frac1{\ubar\alpha\sqrt{\sigma_2}} + \e^{-c_{8}\ubar}\big\} \ll \frac{\Psi(x, y)}{\ubar} .\]
Therefore, the estimate~\eqref{eq:S-objectif} for~$0\leq v\leq v_1$ is implied by the trivial case~$v=0$, and we can suppose from now on that~$v\geq v_1$.

\subsubsection*{Applying the Perron formula}

For all~$n\in S(x^{1/2}\e^{-\gamma}, y)$, we apply Lemma~\ref{lemme-perron} at the abscissa
$$ \sigma = \tfrac12(\beta(v) + \beta(-v)) $$
and height~$\hT$, with the choices
$$ x\gets x/n, \qquad a_d\gets \frac{\bfUn_{P(n) \leq y \text{ and } n\e^{2\gamma} \leq d}}{\tau(nd)}, $$
which yields
\begin{align*} \ssum{P(d) \leq y \\ n\e^{2\gamma}\leq d\leq x/n }\frac1{\tau(nd)} =&\
\frac1{2\pi i}\int_{\sigma-i\hT}^{\sigma+i\hT}\ssum{P(d)\leq y \\ n\e^{2\gamma} \leq d}\frac{x^s}{\tau(nd)(nd)^s}\frac{\dd s} s
+ O\Big(\frac1{\sqrt{\hT}}\ssum{P(d)\leq y \\ n\e^{2\gamma}\leq d} \frac{x^{\sigma}}{\tau(nd)(nd)^{\sigma}}\Big) \\
&\ + O\Big(\frac1{\sqrt{\hT}}\int_{-\sqrt{\hT}}^{\sqrt{\hT}} 
\Big(\ssum{P(d)\leq y \\  n\e^{2\gamma}\leq d} \frac{x^{\sigma+i\tau}}{\tau(nd)(nd)^{\sigma+i\tau}}\Big)K(\tau/\sqrt{\hT})\dd \tau \Big)
.\end{align*}
Let~$\kappa := \beta(v) - \beta(-v)>0$. We sum the previous estimate over all~$y$-friable~$d\geq 1$. By using Rankin's trick~$\bfUn_{n\e^{2\gamma} \leq d} \leq \e^{-\kappa\gamma}(d/n)^{\kappa/2}$ on the error terms, we obtain
\begin{align*}
S(x, y ; \gamma) =&\ \frac1{2\pi i}\int_{\sigma-i\hT}^{\sigma+i\hT} \sum_{P(d)\leq y} \sum_{n\in S(d\e^{-2\gamma}, y)} \frac{x^s}{\tau(nd)(nd)^s}
\frac{\dd s} s + O\Big(\frac{x^{(\beta_1+\beta_2)/2} \e^{\gamma(\beta_1-\beta_1)} F_y(\beta_1, \beta_2)}{\sqrt{\hT}}\Big) \\
&\ + O\Big(\frac{x^{(\beta_1+\beta_2)/2}}{\sqrt{\hT}}\int_{-\sqrt{\hT}}^{\sqrt{\hT}}F_y(\beta_1+i\tau, \beta_2+i\tau)K(\tau/\sqrt{\hT})\dd\tau\Big)
\end{align*}
where, here and in what follows, we abbreviate
$$ (\beta_1, \beta_2) = (\beta(v), \beta(-v)). $$
Next we express the inner sum over~$n$ in the main term using again Lemma~\ref{lemme-perron}, at the abscissa~$\kappa/2$ and the height~$\hT/2$. By similar calculations and using
\[ \Bigg|\int_{\sigma-i\hT}^{\sigma+i\hT}\Big(\frac{x}{nd}\Big)^s\frac{\dd s}s\Bigg| \ll \Big(\frac{x}{nd}\Big)^{\sigma}, \]
as well as the triangle inequality, we obtain
\begin{equation}\label{expr-S-perron}
S_2 = M + O(R_1 + R_2 + R_3),
\end{equation}
where
\begin{align*}
M :=&\ \frac1{(2\pi i)^2}\int_{\sigma-i\hT}^{\sigma+i\hT}\int_{\kappa-i\hT}^{\kappa+i\hT}x^{s}\e^{-\gamma w}F_y(s+w/2, s-w/2)\frac{\dd w}w \frac{\dd s}s, \numberthis\label{def-M-S}\\
R_1 :=&\ \frac{x^{(\beta_1+\beta_2)/2}\e^{\gamma(\beta_2-\beta_1)}F_y(\beta_1, \beta_2)}{\sqrt{\hT}} \ll x^{(\beta_1+\beta_2)/2}\e^{\gamma(\beta_2-\beta_1)}F_y(\beta_1, \beta_2)\hT^{-c_{9}}, \\
R_2 :=&\ \frac{x^{(\beta_1+\beta_2)/2}\e^{\gamma(\beta_2-\beta_1)}}{\sqrt{\hT}}\int_{-\sqrt{\hT}}^{\sqrt{\hT}}|F_y(\beta_1+i\tau, \beta_2+i\tau)|\dd \tau, \\
R_3 :=&\ \frac{x^{(\beta_1+\beta_2)/2}\e^{\gamma(\beta_2-\beta_1)}}{\sqrt{\hT}}\int_{-\sqrt{\hT}}^{\sqrt{\hT}}|F_y(\beta_1-it/2, \beta_2+it/2)|\dd t.
\end{align*}

\subsubsection*{First truncation}

By Corollary~\ref{decr-F-int}.(i), we have
\[ R_1 + R_2 + R_3 \ll x^{(\beta_1+\beta_2)/2}\e^{\gamma(\beta_2-\beta_1)}F_y(\beta_1, \beta_2)\big\{\hT^{-1/2}+\e^{-c_{10}\ubar}\big\} \ll x^{(\beta_1+\beta_2)/2}\e^{\gamma(\beta_2-\beta_1)}F_y(\beta_1, \beta_2)\hT^{-c_{11}}. \]
By~\eqref{taylor-E}, we obtain
\begin{equation}\label{majo-R123}
R_1 + R_2 + R_3 \ll x^\alpha \zeta(\alpha, y)\e^{\cR(v)}T^{-c_{11}} \ll \Psi(x, y)\frac{\e^{\cR(v)}}{\ubar}.
\end{equation}

\subsubsection*{Second truncation}

We now consider~$M$, defined at~\eqref{def-M-S}. Let~$c_3$ be the absolute constant given by Corollary~\ref{decr-F}, and put~$T_1:=c_3/(2\log y)$. We write the integration domain in the double integral~\eqref{def-M-S} as the disjoint union
\[ [-\hT, \hT]^2 = D_1\sqcup D_2,\]
where
\[ D_1 := \{(\tau, t) : |\tau-t/2|\leq T_1 \text{ and } |\tau+t/2|\leq T_1\}, \]
\[ D_2 := \{(\tau, t) : \max\{|\tau|, |t|\}\leq \hT, \text{ and }|\tau-t/2|>T_1\text{ or } |\tau+t/2|> T_1)\}. \]
Accordingly, we call~$I_1$ the contribution of~$(\tau, t)\in D_1$ to~\eqref{def-M-S}, and~$I_2$ the contribution of~$D_2$, so that~$M = I_1 + I_2$. The hypotheses of Corollary~\ref{decr-F-int} are satisfied for~$I_2$, with the parameters
$$ (X, \delta, \lambda_1, \lambda_2, \mu_1, \mu_2) = 
(\hT, c_3/2, -1/2, 1/2, \beta_1+\beta_2, \beta_1-\beta_2). $$
We deduce
\begin{align*}
I_2 &\ \ll x^{(\beta_1+\beta_2)/2}\e^{\gamma(\beta_2-\beta_1)}F_y(\beta_1, \beta_2)(\log(\hT(\log x)/(\alpha^2v_1))^2(\log x)^{-1}H(\ubar)^{-c_{12}} \\
&\ \ll_\ee x^{(\beta_1+\beta_2)/2}\e^{\gamma(\beta_2-\beta_1)}F_y(\beta_1, \beta_2)\hT^{-c_{13}}
\end{align*}
since~$\beta_1-\beta_2 \gg v_1\sqrt{\ubar}/\log x$ and~$\log(1/v_1)\ll \log\log x$. As for~\eqref{majo-R123}, we conclude that
\begin{equation}
I_2 \ll \Psi(x, y)\frac{\e^{\cR(v)}}{\ubar} \label{majo-S2-I2}.
\end{equation}

\subsubsection*{Bounds for large~$v$}

By the change of variables~$(s, w) \gets (s+w/2, s-w/2)$, we write
\begin{align*}
I_1(v) := \frac2{(2\pi i)^2}\int_{\beta_1-iT_1}^{\beta_1+iT_1}\int_{\beta_2-iT_1}^{\beta_2+iT_1}x^{(s+w)/2}\e^{\gamma(w-s)}F_y(s, w)\frac{\dd w \dd s} {(s-w)(s+w)}.
\end{align*}
We first give a rough bound for~$I(v_m)$. Consider first, then, that~$v=v_m$. By the triangle inequality,
$$ I_1(v_m) \ll x^{(\beta_1 + \beta_2)/2}\e^{\gamma(\beta_2-\beta_1)}\int_{-T_1}^{T_1}\int_{-T_1}^{T_1} |F_y(\beta_1+i\tau, \beta_2+it)|\frac{\dd \tau\dd t}{(\beta_1-\beta_2)\alpha} \qquad(v=v_m).$$
From Corollary~\ref{decr-F} and~\eqref{HT-p281}, we have
$$ \int_{-T_1}^{T_1}\int_{-T_1}^{T_1} |F_y(\beta_1+i\tau, \beta_2+it)|\frac{\dd \tau\dd t}{(\beta_1-\beta_2)\alpha} \ll \frac{F_y(\beta_1, \beta_2)}{\alpha\kappa \sigma_2} \qquad (v=v_m). $$
Since we have~$\kappa\asymp v_m/\sqrt{\sigma_2}$, we conclude that
\begin{equation}
I_1(v_m) \ll \frac{x^\alpha\zeta(\alpha, y)\e^{\cR(v_m)}}{v_m\alpha\sqrt{\sigma_2}} \asymp \Psi(x, y) \frac{\e^{\cR(v_m)}}{v_m}.\label{eq:majo-I1-v2}
\end{equation}

\subsubsection*{Differentiation, third truncation}

Now we let~$v$ vary inside~$[v_1, v_m]$. For all such~$v$, the quantity~$I_1$ is differentiable at~$v$ and
$$ I_1'(v) = -\frac{2\cu}{(2\pi i)^2}\int_{\beta_1-iT}^{\beta_1+iT}\int_{\beta_2-iT}^{\beta_2+iT} x^{(s+w)/2}\e^{\gamma(w-s)}F_y(s, w)\frac{\dd w\dd s}{s+w}. $$
Let~$T_0 := (\ubar)^{2/3}/(u\log y)$. We split the previous integrals as
\[ I'_1(v) = J_0(v) + \tilde{J}_0(v), \]
where~$J_0$ is the integral over the box~$D_0 := \{(\tau, t) : |\tau|, |t|\leq T_0\}$, and~$\tilde{J}_0$ is the complementary contribution.

\subsubsection*{Taylor range}

When~$(\tau, t)\in D_0$, we Taylor expand the integrand~: the calculations are very much similar to~\cite[page~280]{HT86}. Letting
\[ Q(\tau, t) := \frac{\tau^2}2\partial_{20}f_y(\beta_1, \beta_2) + \tau t\partial_{11}f_y(\beta_1, \beta_2) + \frac{t^2}2\partial_{02}f_y(\beta_1, \beta_2), \]
we have by a Taylor expansion at order 4, using Lemma~\ref{lemme-props-fy}.(i) and~(iv),
\[ f_y(s, w) = f_y(\beta_1, \beta_2) + i\tau\partial_{10}f_y(\beta_1, \beta_2) + it\partial_{01}f_y(\beta_2, \beta_2) - Q(\tau, t) + \sum_{j=0}^3\lambda_j\tau^jt^{3-j} + O((|\tau|+|t|)^4\sigma_4) \]
for some coefficients~$\lambda_j\ll\sigma_3$. Since~$T_0^4\sigma_4$ and~$T_0^3\sigma_3$ are~$O(1)$, we have
\[ \exp\Big\{\sum_{j=0}^3\lambda_j\tau^jt^{3-j}+O((|\tau|+|t|)^4\phi_4(\alpha, t))\Big\} = 1+\sum_{j=0}^3\lambda_j\tau^jt^{3-j} + O\big((|\tau|+|t|)^6\sigma_3^2 + (|\tau|+|t|)^4\sigma_4\big). \]
Moreover, 
\[ \frac1{s + w} = \frac1{\beta_1+\beta_2} \Big(1 - i\frac{t+\tau}{\beta_1+\beta_2}
+ O\big(\big(\frac{|\tau|+|t|}{\alpha}\big)^{2}\big)\Big). \]

Since we have~$\sigma_3\alpha^{-1} \ll \sigma_4$, we obtain for some numbers~$\mu_1, \mu_2$ independent of~$\tau$ and~$t$,
\begin{equation}
\label{intgd-I1-S2}
\begin{aligned} \frac{x^{(s+w)/2}\e^{\gamma(w-s)}F_y(s, w)}{s + w} &\ = \frac{x^{(\beta_1+\beta_2)/2}\e^{\gamma(\beta_2-\beta_1)}F_y(\beta_1, \beta_2)}{\beta_1 + \beta_2}\e^{-Q(\tau, t)} \Big\{1 + \sum_{j=0}^3\lambda_j\tau^jt^{3-j} \\ &  + \mu_1\tau+\mu_2 t + O\Big( (|\tau|+|t|)^6\sigma_3^2 + (|\tau|+|t|)^4\sigma_4 + (|\tau|+|t|)^2\alpha^{-2}\Big)\Big\}.
\end{aligned}
\end{equation}

Note that we have the formul{\ae} (see also~\cite[formula~(11.13)]{RT}):
\[ \iint_{(\tau, t)\in\bfR^2}\e^{-Q(\tau, t)}\dd\tau\dd t = \frac{2\pi}{\sqrt{\hess[f_y](\beta_1, \beta_2)}} \asymp \frac1{\sigma_2}, \]
\begin{equation}
\iint_{(\tau, t)\in\bfR^2}|\tau|^k|t|^\ell\e^{-Q(\tau, t)}\dd\tau\dd t \ll_{k, \ell} \frac1{(\sigma_2)^{1+(k+\ell)/2}}, \qquad (k, \ell\in\bfN).\label{majo-abs-tau-t}
\end{equation}
We integrate the quantity~\eqref{intgd-I1-S2} over the square~$D_0$. By the symmetry of~$D_0$, the contribution of terms involving~$\lambda_j$ and~$\mu_j$ vanishes. Therefore, using~\eqref{majo-abs-tau-t}, we have
\[ J_0 = -2\cu\Big\{1+O\Big(\frac1{\ubar}\Big)\Big\}\frac{x^{(\beta_1+\beta_2)/2}\e^{\gamma(\beta_2-\beta_1)}F_y(\beta_1, \beta_2)} {(2\pi)^2(\beta_1 + \beta_2)}\iint_{(\tau, t)\in D_0}\e^{-Q(\tau, t)}\dd\tau\dd t. \]
On the other hand, following again~\cite[section 11.3]{RT}, we have
\begin{equation}
\label{queue-gauss}
\begin{aligned}
\iint_{\substack{\max\{|t|, |\tau|\}\geq T_0}}& \e^{-Q(\tau, t)}\dd\tau\dd t \\ & \leq \iint_{\substack{\max\{|t|, |\tau|\}\geq T_0}}\exp\Big\{-\frac{\hess[f_y]}4\Big(\frac{\tau^2}{\partial_{02}f_y} + \frac{t^2}{\partial_{20}f_y}\Big)\Big\}\dd\tau\dd t  
\end{aligned}
\end{equation}
where the partial derivatives of~$f_y$ are evaluated at~$(\beta_1, \beta_2)$. By Lemma~\ref{lemme-props-fy}, we have that~$\hess[f_y](\beta_1, \beta_2) \asymp \sigma_2^2$ and~$\partial_{20}(\beta_1, \beta_2)\ll \sigma_2$ (similarly for~$\partial_{02}f_y$). Therefore, for some~$c>0$, the right-hand side of~\eqref{queue-gauss} is
\[ \leq \iint_{\substack{\max\{|t|, |\tau|\}\geq T_0}}\e^{-c\sigma_2(\tau^2+t^2)}\dd\tau\dd t \ll \frac{\e^{-c\sigma_2T_0^2}}{T_0\sigma_2^{3/2}}.\]
Since~$T_0^2\sigma_2\asymp \ubar^{1/3}$, the above is certainly~$O((\sigma_2\ubar)^{-1})$. We conclude that
\[ J_0 = -2\cu\Big\{1+O\Big(\frac1{\ubar}\Big)\Big\}\frac{x^{(\beta_1+\beta_2)/2}\e^{\gamma(\beta_2-\beta_1)}F_y(\beta_1, \beta_2)} {2\pi(\beta_1+\beta_2)\sqrt{\hess[f_y](\beta_1, \beta_2)}}. \]
Now by~equations~\eqref{taylor-hess} and~\eqref{taylor-beta}, we have
\[ \hess[f_y](\beta_1, \beta_2) = \hess[f_y](\alpha, \alpha) \Big\{1+O\Big(\frac{v^2}{\ubar}\Big)\Big\}, \]
$$ \beta_1 + \beta_2 = 2\alpha + O\Big(\frac{v^2}{\log x}\Big), $$
and by definition of~$\cu$ we have~$\hess[f_y](\alpha, \alpha)=\cu^2\phi_2(\alpha, y)$. We obtain
\begin{equation}
\label{estim-S2-I0}
\begin{aligned}
J_0 =\ & -\Big\{1+O\Big(\frac{1+v^2}{\ubar}\Big)\Big\}\frac{x^{(\beta_1+\beta_2)/2}\e^{\gamma(\beta_2-\beta_1)}F_y(\beta_1, \beta_2)} {2\pi\sqrt{\sigma_2}\alpha} \\
=\ & - \Psi(x, y)\Big\{1+O\Big(\frac{1+v^2}{\ubar}\Big)\Big\}\frac{\e^{\cR(v)}}{\sqrt{2\pi}}
\end{aligned}
\end{equation}
by~\eqref{taylor-E} and~\eqref{estim-psi-col}. This is our expected main term. We note that
\begin{equation}\label{crude-I0-S2} J_0 \asymp \Psi(x, y)\e^{\cR(v)}. \end{equation}

\subsubsection*{Bounds away from the Taylor range}

It remains to estimate~$\tilde{J}_0$, which is the contribution to~$I_1'(v)$ of those~$(\tau, t)$ which satisfy~$\max\{|\tau|, |t|\}\geq T_0$. By Corollary~\ref{decr-F}.(ii) and symmetry, we have
\[ {\tilde J}_0 \ll \sqrt{\sigma_2}\frac{x^{(\beta_1+\beta_2)/2}\e^{\gamma(\beta_2-\beta_1)}F_y(\beta_1, \beta_2)}{\alpha} \int_0^{T_1}\int_{T_0}^{T_1}\Big\{\Big(1+\frac{\tau^2\sigma_2}{y/\log y}\Big)\Big(1+\frac{t^2\sigma_2}{y/\log y}\Big)\Big\}^{-cy/\log y}\dd t\dd\tau. \]
By~\eqref{HT-p281}, we have
\begin{equation}
{\tilde J}_0 \ll \frac{x^{(\beta_1+\beta_2)/2}\e^{\gamma(\beta_2-\beta_1)}F_y(\beta_1, \beta_2)}{\ubar\alpha\sqrt{\sigma_2}} \asymp \Psi(x, y) \frac{\e^{\cR(v)}}{\ubar} \asymp \frac{J_0}{\ubar} \label{majo-S2-It0}
\end{equation}
by~\eqref{taylor-E} and~\eqref{estim-psi-col}. Regrouping our estimates~\eqref{majo-S2-It0} and~\eqref{estim-S2-I0}, we conclude that
\begin{equation}
\frac{I'_1(v)}{\Psi(x, y)} = -\frac{\e^{\cR(v)}}{\sqrt{2\pi}} + O\Big(\frac{1+v^2}{\ubar}\e^{\cR(v)}\Big) \qquad (v_1\leq v\leq v_m).\label{eq:estim-I1p}
\end{equation}

\subsubsection*{Integration}

Since~$I_1(v) = I_1(v_m) - \int_v^{v_m}I_1'(z)\dd z$, estimates~\eqref{eq:estim-I1p} and~\eqref{eq:majo-I1-v2} imply
\begin{equation}
\label{eq:estim-I1}
\begin{aligned}
\frac{I_1(v)}{\Psi(x, y)} = \int_v^{v_m}\e^{\cR(z)}\frac{\dd z}{\sqrt{2\pi}} +
O\Big(\frac{\e^{\cR(v_m)}}{v_m} + \frac1{\ubar}\int_v^{v_m} (1+z^2)\e^{\cR(z)}\dd z\Big)
\end{aligned}
\end{equation}
We regroup estimate~\eqref{eq:estim-I1} with~\eqref{majo-S2-I2} and~\eqref{majo-R123} to obtain the required result
\[ \frac{S}{\Psi(x, y)} = \int_v^{v_m}\e^{\cR(z)}\frac{\dd z}{\sqrt{2\pi}} +
O\Big(\frac{\e^{\cR(v_m)}}{v_m} + \frac{\e^{\cR(v)}}{\ubar^2} + \frac1{\ubar}\int_v^{v_m} (1+z^2)\e^{\cR(z)}\dd z\Big). \]

\vspace{1em}

\appendix

\section{Proof of Lemma~\ref{decr-fy}}

\begin{proof}[Proof of part~(i) of Lemma~\ref{decr-fy}]\renewcommand{\qedsymbol}{}
We shall actually prove that for all~$p\leq y$,
\[ \Re\Big\{\Xi(p^{-s}, p^{-w}) - \Xi(p^{-\sigma}, p^{-\kappa})\Big\} \leq -\frac1{100}\Big\{\log\Big(1+\Big(\frac{\tau}{\sigma}\Big)^2\Big) + \log\Big(1+\Big(\frac{t}{\kappa}\Big)^2\Big) \Big\} \]
from which the lemma follows by summing over~$p\leq y$. For this proof, it will be more convenient to depart slightly from the notation used in the rest of this paper. We put
\[ \mu_1 \gets \sigma\log p, \qquad \tau \gets \tau/\sigma, \]
\[ \mu_2 \gets \kappa\log p, \qquad t \gets t/\kappa. \]
Our objective is to prove that
\begin{equation}\label{eq-obj-1} \Re\Big\{\Xi(\e^{-\mu_1(1+i\tau)}, \e^{-\mu_2(1+it)}) - \Xi(\e^{-\mu_1}, \e^{-\mu_2})\Big\} \leq -\frac1{100}\Big\{\log\big(1+\tau^2\big) + \log\big(1+t^2\big) \Big\} \end{equation}
under the hypotheses
\begin{equation}\label{hyp-mu} |\mu_1-\mu_2|\leq \eta\mu_1, \qquad \max\{\mu_1(1+|\tau|), \mu_2(1+|t|)\} \leq \eta. \end{equation}
We abbreviate further
\[ \bx := \e^{-\mu_1(1+i\tau)}, \quad \by := \e^{-\mu_2(1+it)}, \quad \bz := (1-\bx)/(1-\by). \]
The equation~\eqref{eq-obj-1} is trivially satisfied at~$\tau=t=0$. Replacing~$\tau$ by~$\lambda\tau$ and~$t$ by~$\lambda t$ with~$\lambda$ varying in~$[0, 1]$, and then differentiating with respect to~$\lambda$, we have that it will be sufficient to establish that under the same hypotheses,
\begin{equation}\label{eq-obj-2} \Im\Big\{\mu_1\tau \bx\partial_{10}\Xi(\bx, \by) + \mu_2t \by\partial_{01}\Xi(\bx, \by)\Big\} \leq -\frac1{50}\Big\{\frac{\tau^2}{1+\tau^2} + \frac{t^2}{1+t^2} \Big\} .\end{equation}
The quantity on the LHS can be written as
\[ S := \Im\Big\{-\frac{\mu_1\tau\bx}{1-\bx}\Big(\frac1{\log \bz} - \frac \bz{\bz-1}\Big) -\frac{\mu_2t \by}{1-\by}\Big(\frac 1{\bz-1} - \frac1{\log \bz}\Big) \Big\} \]
where the parentheses involving~$\bz$ are interpreted as~$-1/2$ if~$\bz=1$. Let~$\fr(\bz), \fs(\bz) \in \bfR$ be defined for~$z\in\CpR$ by
\begin{equation}\label{def-frfs} \fr(\bz) + i\fs(\bz) := -1/2+\frac1{\log \bz}-\frac1{\bz-1}. \end{equation}
Note that~$\fr(1/\bz)=-\fr(\bz)$ and~$\fs(1/\bz)=-\fs(\bz)$. We can write~$S = S_1 + S_2$, where
\[ S_1 = \Im\Big\{\frac{-\mu_1\tau \bx}{1-\bx}\Big\}\big(-1/2+\fr(\bz)\big) + \Im\Big\{\frac{-\mu_2 t \by}{1-\by}\Big\}\big(-1/2-\fr(\bz)\big), \]
\[ S_2 = \fs(\bz)\Re\Big\{\frac{-\mu_1\tau \bx}{1-\bx}+\frac{\mu_2t \by}{1-\by}\Big\}. \]
Given the hypotheses~\eqref{hyp-mu}, we remark that we have
\begin{equation}\label{estim-z} \bz = \big\{1+O(\eta)\}\frac{1+i\tau}{1+it}, \end{equation}
\begin{equation}\label{estim-Im} \Im\Big\{\frac{\mu_1\tau \bx}{1-\bx}\Big\} = \frac{-\mu_1\tau\e^{\mu_1}\sin(\mu_1\tau)}{|\e^{\mu_1(1+i\tau)}-1|^2} = -\big\{1+O(\eta)\big\}\frac{\tau^2}{1+\tau^2}, \end{equation}
\begin{equation}\label{estim-Re} \Re\Big\{\frac{\mu_1\tau \bx}{1-\bx}\Big\} = \frac{\mu_1\tau\big(\e^{\mu_1}\cos(\mu_1\tau)-1\big)}{|\e^{\mu_1(1+i\tau)}-1|^2}  = \big\{1+O(\eta(1+|\tau|))\big\}\frac{\tau}{1+\tau^2}. \end{equation}
By symmetry we have the analog estimates for~$\mu_2 t\by/(1-\by)$. We shall use the following inequalities concerning the functions~$\fr$ and~$\fs$.
\begin{lemme}\label{lemme-rs}
When~$\bz\in\bfC\setminus\bfR_-$, the following bounds hold.
\begin{enumerate}[(a)]
\item $|\fr(\bz)|\leq 1/2$,
\item $\fr(\bz) \leq 0 \qquad\qquad \text{for } |\bz|\geq 1$,
\item $|\fr(\bz)|\leq 1/10 \qquad \text{for }1/2\leq |\bz|\leq 2$,
\item $|\fs(\bz)|\leq 1/\pi$,
\item $|\fs(\bz)|\leq 0.15\times|\arg \bz|$.
\end{enumerate}
\end{lemme}
Let us first deduce from these the desired inequality~\eqref{eq-obj-2}. Assume first that one of~$|\tau|$ or~$|t|$ is greater than~$3/2$. By symmetry we suppose that~$|t|\leq |\tau|$ and~$|\tau|\geq 3/2$. Then~\eqref{estim-z} implies that~$|\bz|\geq 1 + O(\eta)$, and since the derivative~$\fr'(\bz)$ is uniformly bounded if~$|\bz|\geq 1/2$, Lemma~\ref{lemme-rs}.(b) yields~$\fr(\bz)\leq O(\eta)$ assuming~$\eta$ is small enough. Using~\eqref{estim-Im} and Lemma~\ref{lemme-rs}.(a), it follows that
\[ S_1 = \big\{1+O(\eta)\big\}\frac{\tau^2}{1+\tau^2}\big(-\frac12 + \fr(\bz)\big) + \big\{1+O(\eta)\big\}\frac{t^2}{1+t^2}\big(-\frac12-\fr(\bz)\big) \leq -\big\{\frac12+O(\eta)\big\}\frac{\tau^2}{1+\tau^2}. \]
On the other hand, using~\eqref{estim-Re}, ~Lemma~\ref{lemme-rs}.(d) and~$|\tau|\geq 3/2$, we have
\begin{align*} S_2 &\leq \frac1{\pi}\Big\{\frac{|\tau|}{1+\tau^2} + \frac{|t|}{1+t^2}(1+O(\eta)) + O\Big(\eta\Big\{\frac{\tau^2}{1+\tau^2}+\frac{t^2}{1+t^2}\Big\}\Big)\Big\} \\
&\leq \frac{1+O(\eta)}{\pi}\Big\{\frac{|\tau|}{1+\tau^2} + \frac12\Big\} + O\Big(\eta\frac{\tau^2}{1+\tau^2}\Big) \\
&\leq \big\{0.45 + O(\eta)\big\}\frac{\tau^2}{1+\tau^2}. \end{align*}
Adding the bounds for~$S_1$ and~$S_2$, we obtain
\[ S \leq \{-0.05 + O(\eta) \}\frac{\tau^2}{1+\tau^2} \leq \{-0.025+O(\eta)\}\Big(\frac{\tau^2}{1+\tau^2}+\frac{t^2}{1+t^2}\Big) \]
which is acceptable granted that~$\eta$ is sufficiently small. Assume now~$|\tau|, |t|\leq 3/2$. Then~$|\bz| \leq \{1+O(\eta)\}\sqrt{1+(3/2)^2} \leq 2$ if~$\eta$ is small enough; similarly~$|\bz|\geq 1/2$. Then proceeding as before, but this time using Lemma~\ref{lemme-rs}.(c), we obtain
\[ S_1 \leq \big\{-0.4 + O(\eta)\big\}\Big(\frac{\tau^2}{1+\tau^2}+\frac{t^2}{1+t^2} \Big). \]
On the other hand, we have
\begin{align*}
  |\arg(\bz)| = \Big|\arg\Big(\frac{1-\bx}{1-\by}\Big)\Big| =\ &
  \Big|\arctan\Big(\frac{\sin(\mu_1\tau)}{\e^{\mu_1}-\cos(\mu_1\tau)}\Big)
  - \arctan\Big(\frac{\sin(\mu_2t)}{\e^{\mu_2}-\cos(\mu_2t)}\Big)\Big| \\
  \leq\ & \arctan(|\tau|) + \arctan(|t|) .
\end{align*}
Using~\eqref{estim-Re} and~Lemma~\ref{lemme-rs}.(e), we have
\begin{align*} S_2 &\leq \big\{0.15+O(\eta)\big\}\big(\arctan(|\tau|)+\arctan(|t|)\big)\Big(\frac{|\tau|}{1+\tau^2}+\frac{|t|}{1+t^2}\Big) \\
&\leq \big\{0.3+O(\eta)\big\}\Big(\frac{\tau^2}{1+\tau^2}+\frac{t^2}{1+t^2}\Big). \end{align*}
We finally obtain in this case
\[ S \leq \big\{-0.1 + O(\eta)\big\}\Big(\frac{\tau^2}{1+\tau^2}+\frac{t^2}{1+t^2}\Big) \]
which is once again acceptable.

To conclude the proof of part~(i) of Lemma~\ref{decr-fy}, it remains to justify Lemma~\ref{lemme-rs}.
\begin{proof}[Proof of Lemma~\ref{lemme-rs}]\renewcommand{\qedsymbol}{}
Since~$\fr(1/\bz)=-\fr(\bz)$ and~$\fr(\bar{\bz})=\overline{\fr(\bz)}$, and similarly for~$\fs(\bz)$, it suffices to consider the case when~$|\bz|\geq 1$ and~$\arg(\bz)\geq 0$. For all~$\omega\in\bfC$, $|\Im\omega|\leq \pi/2$, let
\begin{equation}\label{def-Lo} L(\omega) := \frac1{(\sinh\omega)^2}-\frac1{\omega^2} = \sum_{k\in\bfZ\setminus\{0\}} \frac1{(\omega+ik\pi)^2}.\end{equation}
We shall prove that
\begin{equation}\label{majo-L}
\left\{\begin{aligned}-0.6\leq &\Re L(\omega) \leq 0, \\ &\Im L(\omega) \geq 0,\end{aligned}\right. \qquad (0\leq\Re\omega,\quad 0\leq\Im\omega\leq\pi/2.)
\end{equation}
Let us first prove that~\eqref{majo-L} implies Lemma~\ref{lemme-rs}. Let~$\bz_1, \bz_2\in\CpR$. We have
\begin{align*} \fr(\bz_2)-\fr(\bz_1) &= \frac12\Re\Big\{\int_{(\log \bz_1)/2}^{(\log \bz_2)/2}L(\omega)\dd\omega\Big\} \\&= \frac12\int_{(\log|\bz_1|)/2}^{(\log|\bz_2|)/2}\Re L(t+i\arg(\bz_1)/2)\dd t - \frac12\int_{(\arg \bz_1)/2}^{(\arg \bz_2)/2}\Im L\big((\log |\bz_2|)/2+it\big)\dd t. \end{align*}
Assuming~$1\leq|\bz_1|\leq|\bz_2|$ and~$0\leq\arg \bz_1\leq \arg \bz_2$, the integrals are respectively non-positive and non-negative in view of~\eqref{majo-L}, and we get in this case~$\fr(\bz_1)\geq \fr(\bz_2)$.
By setting~$\bz_1=1$ and~$\bz_2=\bz$, it follows that~$\fr(\bz)\leq 0$. By setting~$\bz_1=\bz$ and~$\bz_2=-X+i0$ and letting~$X\to\infty$, it follows that~$\fr(\bz)\geq \limsup_{X\to\infty}\fr(-X+i0) = -1/2$. Finally, in the case where~$|\bz|\leq 2$, setting~$\bz_1=\bz$ and~$\bz_2 = -2+i0$, we obtain~$\fr(\bz)\geq\fr(-2+i0)$, which evaluates numerically to~$-0.0997\pm10^{-5}\geq -1/10$. This proves parts (a), (b) and (c) of the lemma.
On the other hand, noting that~$L(\omega)\in\bfR$ when~$\omega\in\bfR$, we obtain
\[ \fs(\bz) = \frac12\int_0^{(\arg \bz)/2}\Re L\big((\log |\bz|)/2+it\big)\dd t. \]
The integrand being non-positive, we obtain~$0\geq \fs(\bz)\geq \fs(-|\bz|+i0) = -\frac{\pi}{(\log|\bz|)^2+\pi^2}\geq -1/\pi$ and this proves part~(d) of the lemma. Moreover, by the triangle inequality, we obtain
\[ |\fs(\bz)|\leq \frac{(\arg \bz)}{4}\sup_{\substack{\Re\omega\geq 0\\0\leq \Im\omega\leq \pi/2}}|\Re L(\omega)| \leq 0.15\times(\arg \bz) \]
and this proves part (e).

It remains to prove~\eqref{majo-L}. We recall that~$L$ was defined in~\eqref{def-Lo}. Let~$\omega=a+ib\in\bfC$ be fixed with~$a\geq 0$ and~$0\leq b\leq \pi/2$. We have
\[ \big|\Re L(\omega)\big| \leq \big|L(\omega)\big| \leq \sum_{k\in\bfZ\setminus\{0\}}\frac1{|\omega+ik\pi|^2} \leq \sum_{k\in\bfZ\setminus\{0\}}\frac1{(b+k\pi)^2} = \frac1{(\sin b)^2}-\frac1{b^2}. \]
The last expression is easily seen to be maximal when~$b=\pi/2$; its value at this point is~$1-4/\pi^2 \leq 0.6$.

Next, suppose~$a\leq\pi/2$. Then using the series representation~\eqref{def-Lo}, we have
\[ \Re L(\omega) = \sum_{k\in\bfZ\setminus\{0\}} \frac{a^2-(b+k\pi)^2}{(a^2+(b+k\pi)^2)^2} \leq 0. \]
If on the contrary~$a>\pi/2$, then we have
\[ \Re L(\omega) = \Re\Big\{\frac1{(\sinh \omega)^2}-\frac1{\omega^2}\Big\} = \frac{\sinh^2 a\cos(2b) - \sin^2b}{(\sinh^2a + \sin^2b)^2} - \frac{a^2-b^2}{(a^2+b^2)^2}. \]
This is obviously non-positive if~$\pi/4\leq b\leq \pi/2$. If~$b<\pi/4$, then the above is
\[ \leq \frac{1}{\sinh^2 a} - \frac{1}{a^2}\phi(b/a), \]
where~$\phi(t) := (1-t^2)/(1+t^2)^2$. It is easily verified that~$\phi(t)\geq0.48 \geq (a/\sinh a)^2$ for~$|t|\leq 1/2$ and~$a>\pi/2$. Since indeed~$0\leq b/a\leq 1/2$, it follows that~$\Re L(\omega) \leq 0$ as required.

We turn now to~$\Im L(\omega)$. Consider first the case~$a\leq 4.9$. Using the series representation~\eqref{def-Lo} and grouping indices with same absolute values, we obtain
\begin{equation}\label{eq-ImF-1} \Im L(\omega) = -4ab\sum_{k\geq 1}\frac{a^4+(b^2-(k\pi)^2)(2a^2+b^2+3(k\pi)^2)}{(a^2+(b+k\pi)^2)^2(a^2+(b-k\pi)^2)^2}. \end{equation}
Given that~$b\leq \pi/2$, the numerator is less than~$a^4 + (b^2-\pi^2)(2a^2+b^2+3\pi^2)$; this last expression is maximal when~$b=\pi/2$. At this point, it equals~$a^4 - 3\pi^2/2 a^2 -39\pi^4/16$ which is negative by our assumption that~$0\leq a\leq 4.9$; in view of~\eqref{eq-ImF-1}, we have thus obtained~$\Im L(\omega)\geq 0$ when~$a\leq 4.9$. Suppose now on the contrary that~$a>4.9$. Then
\[ \Im L(\omega) = \frac{-2\sinh a\cosh a\sin b\cos b}{(\sinh^2a + \sin^2b)^2}+\frac{2ab}{(a^2+b^2)^2} = -2\psi\big(\frac{\sin b}{\sinh a}\big)\frac{\cosh a\cos b}{\sinh^2a} + \frac2{a^2}\psi(b/a) \]
where~$\psi(t) := t/(1+t^2)^2$. It is easily established that~$\psi(t)/t\geq 0.81$ for~$0\leq t\leq 1/3$, and~$\psi(t)\leq t$ for~$t>0$. Since~$b/a\leq 1/3$, we obtain
\[ \Im L(\omega) \geq -2\frac{\sin b\cos b\cosh a}{\sinh^3 a} + \frac{1.62\times b}{a^3}. \]
Using~$a>4.9$, we have~$a^3(\cosh a)/\sinh^3 a\leq 0.03$. Consequently,
\[ \Im L(\omega) \geq \frac{b}{a^3}\big\{1.62-0.06\big\} \geq 0. \]
This concludes the proof of~\eqref{majo-L}, hence of part~(i) of Lemma~\ref{decr-fy}.
\end{proof}
\end{proof}
\begin{proof}[Proof of part~(ii) of Lemma~\ref{decr-fy}]\renewcommand{\qedsymbol}{}
We quote from~\cite[equation~(5.2)]{BT-Residu} the bound
$$ \frac{\zeta(s, y)}{\zeta(\sigma, y)} \ll \Big(1 + \frac{\tau}{\sigma\log y}\Big)^{-c_\ee \pi(y)} \qquad (|\tau|\leq \e^{(\log y)^{3/2-\ee}}, 0<\sigma \ll 1/\log y). $$
From the definition~\eqref{definition-h}, we deduce that it will suffice to prove
\begin{equation}
\Re(g(z)) \geq O(1) \qquad (z\in\CpR).\label{eq:mino-Reg}
\end{equation}
Because~$g({\bar z}) = \overline{g(z)}$ and~$g(1/z)=g(z)$, we may assume~$|z|\geq 1$ and~$\arg z \geq 0$. In terms of~$w = 2\log z$, this means~$\Re w \geq 0$,~$\Im w \in [0, \pi/2)$ and we have
$$ \tg(w) := g(\e^{2w}) = \log\Big(\frac{\sinh w}{w}\Big). $$
Notice that, with the definition~\eqref{def-Lo},
$$ \tg''(w) = -L(w). $$
From~\eqref{majo-L}, we deduce
$$ \Re \tg'(w) = -\int_0^{\Re w}\Re L(t)\dd t + \int_0^{\Im w}\Im L(\Re w + it)\dd t \geq 0, $$
$$ \Im \tg'(w) = -\int_0^{\Im w}\Re L(\Re w + it)\dd t \leq 0.6\pi/2\leq 1, $$
so that
$$ \Re g(z) = \Re \tg(w) = \int_0^{\Re w}\Re\tg'(t)\dd t - \int_0^{\Im w}\Im\tg'(\Re w + it)\dd t \geq -\pi/2. $$
Therefore,~\eqref{eq:mino-Reg} holds and so does part~(ii) of Lemma~\ref{decr-fy}.
\end{proof}
\begin{proof}[Proof of part~(iii) of Lemma~\ref{decr-fy}]
First we note that since~$\sigma\leq 0.6$, by the previously done calculation~\eqref{majo-phi2-tail}, we have~$\phi_2(\sigma, y^{1/2}) \ll y^{-0.2}\phi_2(\sigma, y)$ and similarly for~$\kappa$. For large enough~$y$, it will therefore be sufficient to show that under the stated conditions, there exists~$c(\ee)>0$ such that
\[ \Re\Big\{\Xi(p^{-s}, p^{-w}) - \Xi(p^{-\sigma}, p^{-\kappa})\Big\} \leq -c(\ee)\Big\{ \frac{p^{\sigma}(\tau\log p)^2}{(p^\sigma-1)^2} + \frac{p^{\kappa}(t\log p)^2}{(p^\kappa-1)^2}\Big\} \]
for~$y^{1/2}\leq p\leq y$ (so that~$\log p\asymp\log y$). We shall once again depart slightly from the notation used up to now. Relabelling
\[ \bx\gets p^{-s}, \quad \by=p^{-w}\quad\text{and}\quad \tau\gets\arg \bx, \quad t\gets \arg\by, \]
we wish to show that for all~$\ee>0$, there exists~$c=c(\ee)>0$ such that for all~$(\bx, \by)\in\cU^2$ satisfying
\begin{equation}
\max\{|\bx|, |\by|\}\leq 1-\ee, \quad \max\{\tau, t\}\leq c ,\label{eq:proof-lemma5ii-conds}
\end{equation}
we have
\[ \Re\big\{\Xi(\bx, \by) - \Xi(|\bx|, |\by|)\big\} \leq -c \Big\{\frac{|\bx|\tau^2}{(1-|\bx|)^2} + \frac{|\by|t^2}{(1-|\by|)^2}\Big\}. \]
We consider~$\ee$ as being fixed, and let implicit constants depend of~$\ee$ throughout the rest of the proof. Replacing~$\tau$ by~$\lambda\tau$ and~$t$ by~$\lambda t$ and differentiating with respect to~$\lambda$ as before, we see that it will be sufficient to show under the same assumptions~\eqref{eq:proof-lemma5ii-conds}, we have
\[ \Im\big\{\tau \bx \partial_{10}\Xi(\bx, \by) + t \by \partial_{01}\Xi(\bx, \by)\big\} \leq -c\Big\{\frac{|\bx|\tau^2}{(1-|\bx|)^2} + \frac{|\by|t^2}{(1-|\by|)^2}\Big\}. \]
Suppose, then that the conditions~\eqref{eq:proof-lemma5ii-conds} hold for some~$c>0$. In the same way as before, we write the left-hand side as~$S_1 + S_2$, where
\[ S_1 := \Im\Big\{\frac{-\tau\bx}{1-\bx}\Big\}\Big(-\frac12+\fr(\bz)\Big) + \Im\Big\{\frac{-t\by}{1-\by}\Big\}\Big(-\frac12-\fr(\bz)\Big), \]
\[ S_2 := \fs(\bz)\Re\Big\{\frac{-\tau\bx}{1-\bx}+\frac{t \by}{1-\by}\Big\}, \]
and~$\bz=(1-\bx)/(1-\by)$. Here the function~$\fr(\bz)$ and~$\fs(\bz)$ are again defined by~\eqref{def-frfs}. Now we note that
\[ \Im\Big\{\frac{-\tau\bx}{1-\bx}\Big\} = -\frac{|\bx|\tau \sin\tau}{(1-|\bx|)^2} = -\{1+O(c)\}\frac{|\bx|\tau^2}{(1-|\bx|)^2}, \]
\[ \Re\Big\{\frac{\tau\bx}{1-\bx}\Big\} = \frac{|\bx|\tau(\cos\tau - |\bx|)}{(1-|\bx|)^2} = \{1+O(c)\}\frac{|\bx|\tau}{1-|\bx|} \]
granted that~$c$ is small enough in terms of~$\ee$ (note that~$|1-\bx| = (1+O(c))(1-|\bx|)$). Moreover, we have
\[ \arg(1-\bx) = \arctan\Big(\frac{|\bx|\sin\tau}{1-|\bx|\cos(\tau)}\Big) = \{1+O(c)\}\frac{|\bx|\tau}{1-|\bx|} = O(|\bx|\tau). \]
By symmetry the same estimates hold for the analogous quantities for~$\by$. Note that by our hypotheses we have~$|\bz| = O(1)$, which implies
\[ \fr(\bz) = \fr(|\bz|) + O(c), \qquad \fs(\bz) = (\arg \bz)\Big\{\Big[\frac{\partial\ \fs(\bz')}{\partial(\arg \bz')}\Big]_{\bz'=|\bz|} + O(c)\Big\}. \]
Let
\[ \psi_1(\rho) := \fr(\rho) = -\frac12+ \frac1{\log\rho}-\frac1{\rho-1},\quad \psi_2(\rho) := \Big[\frac{\partial\fs(\bz')}{\partial(\arg \bz')}\Big]_{\bz'=\rho} = \frac{\rho}{(\rho-1)^2}-\frac1{(\log\rho)^2}. \]
We regroup the estimates above to obtain
\[ S_1 = \frac{|\bx|\tau^2}{(1-|\bx|)^2}\Big(-\frac12+\psi_1(|\bz|)\Big) + \frac{|\by|t^2}{(1-|\by|)^2}\Big(-\frac12-\psi_1(|\bz|)\Big) + O\big(c(|\bx|\tau^2+|\by|t^2)\big), \]
\[ S_2 = -\psi_2(|\bz|)\Big(\frac{-|\bx|\tau}{1-|\bx|}+\frac{|\by|t}{1-|\by|}\Big)^2 + O(c(|\bx|\tau + |\by|t)^2). \]
Here we used the fact that the quantities~$\fr(|\bz|)$ and~$\fs(\bz)$ are uniformly bounded. Note that~$\psi_2(|\bz|)<0$. By the upper bound~$(\lambda-\mu)^2\leq 2(\lambda^2+\mu^2)\ $ ($\lambda, \mu\in\bfR$), we get
\[ S\leq -\frac{|\bx|\tau^2}{(1-|\bx|)^2}\Big(\frac12-\psi_1(|\bz|)+2|\bx|\psi_2(|\bz|)\Big) -\frac{|\by|t^2}{(1-|\by|)^2}\Big(\frac12+\psi_1(|\bz|)+2|\by|\psi_2(|\bz|)\Big) + O(c(|\bx|\tau^2+|\by|t^2)) .\]
It is easily seen that~$1/2 - |\psi_1(\rho)|  + 2\psi_2(\rho) > 0$ for~$\rho>0$. Since~$|\bz|$ is bounded in terms of~$\ee$, it follows that both parentheses in the above expression are greater than some~$c' = c'(\ee)>0$, so that
\[ S\leq -\big(c'+O(c)\big)\Big\{\frac{|\bx|\tau^2}{(1-|\bx|)^2} + \frac{|\by|t^2}{(1-|\by|)^2}\Big\}. \]
Then, choosing~$c$ sufficiently small relative to~$c'$, we have
\[ S\leq -c\Big\{\frac{|\bx|\tau^2}{(1-|\bx|)^2} + \frac{|\by|t^2}{(1-|\by|)^2}\Big\} \]
as required.
\end{proof}

\bibliographystyle{amsalpha}
\bibliography{arcsinus-friable}

\end{document}